\documentclass[12pt]{article}

\usepackage{amssymb, amsfonts, amsmath}
\usepackage{mathrsfs}
\usepackage{graphicx}
\usepackage[usenames,dvipsnames]{xcolor}
\usepackage{tikz}
\usepackage{float}
\usepackage{enumitem}
\usepackage{comment}
\usepackage{setspace}
\usepackage{caption}
\usepackage{subcaption}
\usepackage[hidelinks]{hyperref}
\usepackage{array}%
\usepackage{multicol}

\newcolumntype{C}[1]{>{\centering\let\newline\\\arraybackslash\hspace{0pt}}m{#1}}
\providecommand{\U}[1]{\protect\rule{.1in}{.1in}}

\newcommand{\Mod}[1]{\ (\text{mod}\ #1)}
\newcommand{\ig}[1]{\mathscr{I}(#1)}
\newcommand{\agraph}{\alpha}
\newcommand{\ag}[1]{\mathscr{A}(#1)}
\newcommand{\cp}[1]{\overline{#1}}
\newcommand{\thet}[1]{\Theta\left\langle #1\right\rangle }

\providecommand{\rad}{\operatorname{rad}}

\newcommand{\Dia}{\mathfrak{D}}
\newcommand{\house}{\mathcal{H}}

\newcommand{\ntni}{\mathfrak{T}}

\newcommand{\dual}[1]{\widetilde{#1}}

\def\ieInner(#1,#2,#3,#4){#1 \overset{#2 #3}{\sim} #4}
\def\edge#1{\ieInner(#1)}
\def\ieInnerA(#1,#2,#3){\overset{#1 #2}{\sim} #3}
\def\adedge#1{\ieInnerA(#1)}
\def\sedge(#1,#2){#1\sim #2}
\def\ieInnerB(#1,#2,#3,#4,#5){#1 \overset{#2 #3}{\sim_{#5}} #4} 
\def\edgeG#1{\ieInnerB(#1)}
\def\ieInnerC(#1,#2,#3){#1 {\sim_{#3}} #2 }

\def\ieInnerD(#1,#2,#3,#4){#1 \sim #2 \sim \dots \sim #3 \sim #4}

\def\ieInnerE(#1,#2){#1 \sim #2}

\def\ieInnerF(#1,#2){\overset{#1 #2}{\sim}}

\definecolor{ltsky}{RGB}{0,191,255}
\definecolor{ltteal}{RGB}{0, 128, 128 }
\definecolor{medOrch}{RGB}{122,55,139}
\definecolor{royalBlue}{RGB}{65,105,225}
\definecolor{forGreen}{RGB}{34,139,34}
\definecolor{dand}{RGB}{255,193,37}
\definecolor{lightBlue}{RGB}{176,226,255}
\definecolor{lgrey}{RGB}{209,209,209}
\definecolor{lgray}{gray}{0.95}
\definecolor{mgray}{gray}{0.40}
\definecolor{mmgray}{gray}{0.60}
\definecolor{zelim}{RGB}{7,163,82}
\definecolor{zelim2}{RGB}{5,114,57}

\usetikzlibrary{fit,positioning,calc, arrows, shapes}
\usetikzlibrary{decorations.pathreplacing}
\usetikzlibrary{backgrounds}
\tikzstyle{std}=[ circle, draw=black,fill=black,thick, inner sep=2pt, minimum size=2.5mm]
\tikzstyle{wstd}=[ circle, draw=black,fill=white,thick, inner sep=2pt, minimum size=2.5mm]
\tikzstyle{ir}=[ circle, draw=black,fill=green,thick,  inner sep=2pt, minimum size=2mm]
\tikzstyle{mp}=[circle, draw=black,fill=Dandelion,thick,  inner sep=2pt, minimum size=2mm]
\tikzstyle{bred}=[circle, draw=black,fill=red,thick,  inner sep=2pt, minimum size=2mm]
\tikzstyle{byellow}=[circle, draw=black,fill=yellow,very thick,  inner sep=2pt, minimum size=3mm]
\tikzstyle{bgray}=[circle, draw=black,fill=mmgray,thick,  inner sep=2pt, minimum size=2mm]
\tikzstyle{bzed}=[circle, draw=black,fill=zelim,thick,  inner sep=2pt, minimum size=2mm]
\tikzstyle{smred}=[ circle, draw=black,fill=red,thick,  inner sep=1pt, minimum size=1.5mm]
\tikzstyle{bblue}=[ circle, draw=black,fill=blue,thick,  inner sep=2pt, minimum size=2.5mm]
\tikzstyle{bgreen}=[ circle, draw=black,fill=green,thick,  inner sep=2pt, minimum size=2.5mm]
\tikzstyle{regRed}=[ circle, draw=black,fill=red,thick,  inner sep=2pt, minimum size=2.5mm]
\tikzstyle{sqRed}=[rectangle, draw=black,fill=red,thick,  inner sep=2pt, minimum size=2.5mm]
\tikzstyle{regG}=[ circle, draw=black,fill=green,thick,  inner sep=2pt, minimum size=2mm]
\tikzstyle{dred}=[diamond, draw=black,fill=red,thick,  inner sep=2pt, minimum size=2mm]
\tikzstyle{tr}=[color=black, style=dotted]
\tikzstyle{sp}=[color=ProcessBlue, line width = 2pt]

\tikzstyle{bline}=[color=blue, line width = 1pt]
\tikzstyle{rline}=[color=red, line width = 1pt]
\tikzstyle{rlinelt}=[color=red, line width = 1pt,opacity=0.2]
\tikzstyle{rlinemb}=[color=red, line width = 1pt, densely dashdotted]
\tikzstyle{gline}=[color=mmgray, line width = 1pt]
\tikzstyle{zline}=[color=zelim, line width = 1pt]
\tikzstyle{ll}=[color=gray]

\tikzstyle{vertex}=[circle, fill=black, inner sep= 0, minimum size = 4]

\tikzset{broadArrow/.style={single arrow, fill=red!50, anchor=base, align=center,text width=2.8cm}}

\pgfdeclarelayer{blob}    
\pgfdeclarelayer{bedge}    
\pgfsetlayers{bedge,blob,main}

\usepackage{authblk}

\newtheorem{theorem}{Theorem}[section]

\newtheorem{cons}[theorem]{Construction}
\newtheorem{coro}[theorem]{Corollary}

\newtheorem{lemma}[theorem]{Lemma}

\newtheorem{obs}[theorem]{Observation}

\newtheorem{problem}{Problem}
\newtheorem{prop}[theorem]{Proposition}

\newenvironment{proof}[1][Proof]{\noindent\textbf{#1.} }{\ \hfill \rule{0.5em}{0.5em}}

\begin{document}

\title{\textbf{The Realizability of Theta Graphs as Reconfiguration Graphs of \linebreak Minimum Independent Dominating Sets}}

\author{R. C. Brewster\thanks{Funded by Discovery Grants from the Natural Sciences and
Engineering Research Council of Canada, RGPIN-2014-04760, RGPIN-03930-2020.}}
\affil{Department of Mathematics and Statistics\\Thompson Rivers University\\805 TRU Way\\Kamloops, B.C.\\ \textsc{Canada} V2C 0C8}
\author{C. M. Mynhardt$^{*}$}
\author{L. E. Teshima}
\affil{Department of Mathematics and Statistics\\University of Victoria\\PO BOX 1700 STN CSC\\Victoria, B.C.\\\textsc{Canada} V8W 2Y2 \authorcr {\small rbrewster@tru.ca, kieka@uvic.ca, lteshima@uvic.ca}}

\maketitle

\begin{abstract}

The independent domination number $i(G)$ of a graph $G$ is the minimum cardinality of a maximal independent set 
of $G$, also 
called an $i(G)$-set.  The $i$-graph of $G$, denoted $\ig{G}$, is the graph whose vertices correspond to the 
$i(G)$-sets, and where two $i(G)$-sets are adjacent if and only if they differ by two adjacent vertices. Not all graphs are $i$-graph realizable, that is, given a target graph $H$, there does not necessarily exist a source graph $G$ such that $H \cong \ig{G}$. 
We consider a class of graphs called ``theta graphs'': a theta graph is the union of three internally disjoint nontrivial paths with the same two distinct end vertices. We characterize theta graphs that are $i$-graph realizable, showing that there are only finitely many that are not. We also characterize those line graphs and claw-free graphs that are $i$-graphs, and show that all $3$-connected cubic bipartite planar graphs are $i$-graphs.
\end{abstract}

\noindent\textbf{Keywords:\hspace{0.1in}}independent domination number, graph reconfiguration, $i$-graph, theta graph

\noindent\textbf{AMS Subject Classification Number 2020:\hspace{0.1in}}05C69	

\section{Introduction}
\label{sec:intro}
Our main topic is the characterization of theta graphs that are obtainable as reconfiguration graphs with respect to the minimum independent dominating sets of some graph.

In graph theory, reconfiguration problems are often concerned with solutions to a specific problem that are
vertex subsets of a graph. When this is the case, the reconfiguration problem can be viewed as 
a token manipulation problem, where
a solution subset is represented by placing a token at each vertex of the subset. Each solution is 
represented as a vertex of a new graph, referred to as a \textit{reconfiguration graph}, where adjacency between vertices corresponds to a predefined token manipulation rule called the \emph{reconfiguration step}. The reconfiguration step we consider here consists of sliding a single token along an edge between adjacent vertices
belonging to different solutions.

More formally, given a graph $G$, the \textit{slide graph} of $G$ is the graph $H$ such that each vertex of $H$ represents a solution of some problem
on $G$, and two vertices $u$ and $v$ of $H$ are adjacent if and only if the solution in $G$ corresponding
to $u$ can be transformed into the solution corresponding to $v$ by sliding a single
token along the edge $uv \in E(G)$. See \cite{MN20} for a survey on reconfiguration of colourings and dominating sets in graphs. 

We use the standard notation of $\alpha(G)$ for the independence number of a graph $G$. 
The \emph{independent domination number} $i(G)$ of $G$ is the
minimum cardinality of a maximal independent set of $G$, or, equivalently,
the minimum cardinality of an independent domination set of $G$. An independent dominating set of $G$ of cardinality $i(G)$ is also called an $i$-\emph{set} of $G$, or an $i(G)$-\emph{set}. An $\alpha$-\emph{set} of $G$, or an $\alpha(G)$-\emph{set}, is defined similarly. When $i(G)=\alpha(G)$, we say that $G$ is \emph{well-covered}.

Given a graph $G$, we consider the slide graph $\ig{G}$ of $G$, formally defined in Section \ref{sec:iGr} below, with respect to its $i$-sets. Our main result, Theorem \ref{thm:c:thetas}, concerns the class $\Theta$ of ``theta graphs'' for which we characterize, in Sections \ref{sec:c:thetaComp} and \ref{sec:c:nonthetas}, those graphs $H$ for which there exists a graph $G \in \Theta$ such that $H \cong \ig{G}$. We state known results required here in Section~\ref{sec:previous}. We introduce the technique we use for theta graphs in Section \ref{sec:c:line}, where we apply it to the simpler problem of characterizing line graphs and claw-free graphs that are realizable as $i$-graphs. In Section \ref{subsubsec:c:nonThetnonI} we exhibit a graph that is neither a theta graph nor an $i$-graph, and in Section \ref{subsubsec:c:maxPlanar} we show that certain planar graphs are $i$-graphs and $\alpha$-graphs. We conclude with some open problems in Section \ref{sec:open}.

In general, we follow the notation of \cite{CLZ}. For other domination principles and
terminology, see \cite{HHS2, HHS1}.

\section{$i$-Graphs and Theta Graphs}
\label{sec:iGr}
The \emph{$i$-graph} of a graph $G$, denoted $\ig{G}=(V(\ig{G}),E(\ig{G}))$, is the graph with vertices 
representing the minimum independent dominating sets of $G$ (that is, the {$i$-sets} of $G$), and where $u,v\in V(\ig{G})$, corresponding to the $i(G)$-sets  $S_u$ and $S_v$, respectively, are 
adjacent in $\ig{G}$ if and only if there exists  $xy \in E(G)$ such that $S_u = (S_v-x)\cup\{y\}$. Imagine that there is a token on each vertex of an $i$-set $S$ of $G$.  
Then $S$ is adjacent, in $\ig{G}$, to an $i(G)$-set $S'$ if and only if a single token can slide along an 
edge of $G$ to transform $S$ into $S'$. 
Similarly, the \emph{$\agraph$-graph}  $\ag{G}$ of a graph $G$ is the slide reconfiguration graph with vertices representing the $\alpha(G)$-sets, and where adjacency is defined as for the $i$-graph.  In Section \ref{sec:c:thetaComp} we present several constructions for $i$-graphs that are also constructions for $\agraph$-graphs.

We say $H$ \emph{is an $i$-graph}, or is $i$-\emph{graph realizable}, if there exists some graph $G$ such that 
$\ig{G}\cong H$. Moreover, we refer to $G$ as the \emph{seed graph} of the $i$-graph $H$.  
Going forward, we mildly abuse notation to denote both the $i$-set $X$ of $G$ and its corresponding 
vertex in $H$ as $X$, so that $X \subseteq V(G)$ and $X \in V(H)$.


In acknowledgment of the slide-action in $i$-graphs, given $i$-sets  $X = \{x_1,x_2,\dots,x_k\}$ and 
$Y=\{y_1,x_2,\dots x_k\}$ of $G$ with $x_1y_1 \in E(G)$, we denote the adjacency of $X$ and $Y$ in $\ig{G}$ 
as $\edge{X,x_1,y_1,Y}$, where we imagine transforming the $i$-set $X$ into $Y$ by sliding the token at $x_1$ 
along an edge to $y_1$.  When discussing several graphs, we use the notation $\edgeG{X,x_1,y_1,Y,G}$ 
to specify that the relationship is on $G$.  
More generally, we use $x \sim y$ to denote the adjacency of vertices $x$ and $y$ (and $x\not \sim y$ to 
denote non-adjacency); this is used in the context of both the seed graph and the target graph.

The study of $i$-graphs was initiated by Teshima in \cite{LauraD}. In \cite{BMT1}, Brewster, Mynhardt and Teshima investigated $i$-graph realizability and proved some results concerning the 
adjacency of
vertices in an $i$-graph and the structure of their associated $i$-sets in the seed graph. They presented the 
three smallest graphs that are not $i$-graphs: the diamond graph $\Dia = K_4-e$, $K_{2,3}$ and the graph
$\kappa$, which is $K_{2,3}$ with an edge subdivided. They showed that several graph classes, like trees
and cycles, are $i$-graphs. They demonstrated that known $i$-graphs can be used to construct new $i$-graphs and 
applied these results to build other classes of $i$-graphs, such as block graphs, hypercubes, forests, 
cacti, and unicyclic graphs.

The diamond $\Dia$, $K_{2,3}$, and $\kappa$ are examples of \emph{theta graphs}: graphs that 
are the union of three internally disjoint nontrivial paths with the same two distinct end vertices.
The graph $\thet{j,k,\ell}$, where $j \leq k \leq \ell$, is the theta graph with paths of lengths $j$, $k$,  and $\ell$.  In this notation, the three non-$i$-graph realizable examples are $\Dia \cong \thet{1,2,2}$, $K_{2,3}\cong\thet{2,2,2}$, and $\kappa \cong \thet{2,2,3}$.


 \section{Previous Results}
 \label{sec:previous}

 To begin, we state several useful observations and lemmas from \cite{BMT1, LauraD} about the structure of $i$-sets within given $i$-graphs. Given a set $S \subseteq V(G)$ and a vertex $v \in S$, the \emph{private neighbourhood} of $v$ with respect to $S$ is the set $\mathrm{pn}(v,S)=N[v]-N[S-\{v\}]$, and the \emph{external private neighbourhood} of $v$ with
respect to $S$ is the set $\mathrm{epn}(v,S)=\mathrm{pn}(v,S)-\{v\}$.

\begin{obs} \label{obs:i:edge} \emph{\cite{BMT1, LauraD}}
	Let $G$ be a graph and $H=\ig{G}$.  A vertex $X \in V(H)$ has $\deg_H(X)\geq 1$ if and only if for some $v \in X \subseteq V(G)$, there exists $u\in \mathrm{epn}(v,X)$ such that $u$ dominates $\mathrm{pn}(v,X)$.  
\end{obs}

For some path $X_1,X_2,\dots,X_k$ in $H$, at most one vertex of the $i$-set is changed at each step, and so $X_1$ and $X_k$ differ on at most $k-1$ vertices.  This is yields the following immediate observation.  

\begin{obs} \label{obs:i:dist} \emph{\cite{BMT1, LauraD}}
	Let $G$ be a graph and $H = \ig{G}$.  Then for any $i$-sets $X$ and $Y$ of $G$, the distance $d_H(X,Y) \geq |X-Y|$.  
\end{obs}


\begin{lemma} \label{lem:i:claw} \emph{\cite{BMT1, LauraD}}
	Let $G$ be a graph with $H = \ig{G}$.  Suppose $XY$ and $YZ$ are edges in $H$ with $\edge{X,x,y_1,Y}$ and $\edge{Y,y_2,z,Z}$, with $X \neq Z$.  Then $XZ$ is an edge of $H$ if and only if $y_1=y_2$.
\end{lemma}

\noindent Combining the results from Lemma \ref{lem:i:claw} with Observation \ref{obs:i:dist} yields the following observation for vertices of $i$-graphs at distance $2$. 

\begin{obs} \label{obs:i:d2} \emph{\cite{BMT1, LauraD}}
	Let $G$ be a graph and $H = \ig{G}$.  Then for any $i$-sets $X$ and $Y$ of $G$, if $d_H(X,Y) = 2$, then $|X-Y|=2$.  
\end{obs}

\begin{lemma} \label{lem:i:K1m}   \emph{\cite{BMT1, LauraD}}
Let $G$ be a graph and  $H = \ig{G}$.  Suppose $H$ contains an induced $K_{1,m}$ with vertex set $\{X,Y_1,Y_2, \dots, Y_m\}$ and $\deg_H(X)=m$.  Let $i \neq j$.  Then in $G$,  
	\begin{enumerate}[itemsep=1pt, label=(\roman*)]
		\item $X-Y_i \neq X-Y_j$, \label{lem:K1m:a}
		\item  $|Y_i \cap Y_j| = i(G)-2$, and  \label{lem:K1m:b}
		\item $m \leq i(G)$. \label{lem:K1m:c}
	\end{enumerate}
\end{lemma}

\begin{prop} \label{prop:i:C4struct}  \emph{\cite{BMT1, LauraD}}
	Let $G$ be a graph and $H=\ig{G}$.  Suppose $H$ has an induced $C_4$ with vertices $X,A,B,Y$, where $XY, AB \notin E(H)$.   Then, without loss of generality, the set composition of $X,A,B,Y$ in $G$, and the edge labelling  of  the induced $C_4$ in $H$, are as in Figure \ref{fig:i:C4struct}.
\end{prop}

\begin{figure}[H]
	\centering
	\begin{tikzpicture}[scale=0.8]
		\node[std,label={270:$Y=\{y_1,y_2,v_3,\dots,v_k\}$}] (y) at (0,0) {};
		\node[std,label={180:$A=\{y_1,x_2,v_3,\dots,v_k\}$}] (a) at (-2,2) {};
		\node[std,label={0:$B=\{x_1,y_2,v_3,\dots,v_k\}$}] (b) at (2,2) {};
		\node[std,label={90:$X=\{x_1,x_2,v_3,\dots,v_k\}$}] (x) at (0,4) {};
		
		\draw[thick] (x)--(a) node[midway,left=10pt,fill=white]{$\edge{X,x_1,y_1,A}$};
		\draw[thick] (x)--(b) node[midway,right=10pt,fill=white]{$\edge{X,x_2,y_2,B}$};
		\draw[thick] (a)--(y) node[midway,left=10pt,fill=white]{$\edge{A,x_2,y_2,Y}$};
		\draw[thick] (b)--(y) node[midway,right=10pt,fill=white]{$\edge{B,x_1,y_1,Y}$};
	\end{tikzpicture}			
	\caption{Reconfiguration structure of an induced $C_4$ subgraph from Proposition \ref{prop:i:C4struct}.}
	\label{fig:i:C4struct}%
\end{figure}
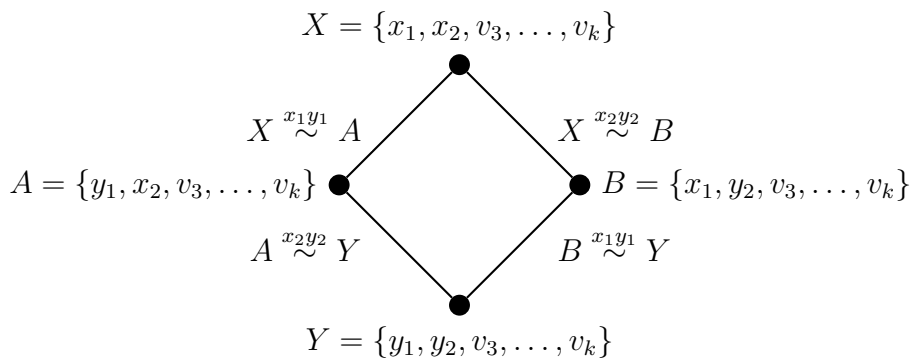

 We state the above-mentioned results regarding $\Dia $, $K_{2,3}$, and $\kappa$ here for referencing.

\begin{prop} \label{prop:i:diamond} \emph{\cite{BMT1, LauraD}}
    The graphs $\Dia $, $K_{2,3}$, and $\kappa$ are not $i$-graph realizable.
 \end{prop}

 On the other hand, the house graph $\house=\thet{1,2,3}$ in Figure \ref{fig:i:houseConstruct}(b) is a theta graph, as illustrated by the seed graph $G$ in Figure  \ref{fig:i:houseConstruct}(a). The $i$-sets of $G$ and their adjacencies are overlaid on $\mathcal{H}$ in Figure  \ref{fig:i:houseConstruct}(b).
\begin{prop}\label{prop:i:house} \emph{\cite{BMT1, LauraD}}
	The house graph $\house$ is an $i$-graph.
\end{prop}

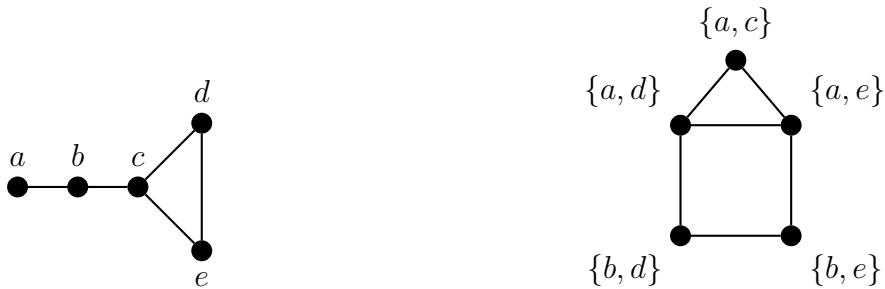
\begin{figure}[H]
	\centering	
	\begin{subfigure}{.4\textwidth}  \centering
		\begin{tikzpicture}	[scale=0.8]		
			\coordinate (cent) at (0,0) {};
			
			\foreach \i / \j in {a/1,b/2,c/3}
			{
				\path(cent) ++(0: \j*10mm) node[std,label={ 90:${\i}$}]  (\i) {};
			}
			\draw[thick] (a)--(c);	
			
			\foreach \i / \j in {d/45,e/-45}
			{
				\path(c) ++(\j: 15 mm) node[std,label={ 2*\j:${\i}$}]  (\i) {};
			}
			
			\draw[thick](c)--(d)--(e)--(c);	
			
			\path (cent) ++ (90:30mm) coordinate (vfil) {};
			
			
			
		\end{tikzpicture}
		\caption{A graph $G$ such that $\ig{G} = \mathcal{H}$.}
		\label{fig:i:houseConstruct:a}%
	\end{subfigure}	\hspace{1cm}
	\begin{subfigure}{.5\textwidth} \centering
		\begin{tikzpicture}	[scale=0.8]		
			\coordinate (cent) at (0,0) {};
			
			\foreach \i / \j / \k in {1/a/e, 2/a/d, 3/b/d, 4/b/e}
			{
				\path(cent) ++(\i*90-45: 13mm) node[std,label={\i*90-45:$\{\j,\k\}$}]  (v\i) {};
			}
			\draw[thick] (v1)--(v2)--(v3)--(v4)--(v1);
			
			\path(cent) ++(90: 20mm) node[std,label={ 90:$\{a,c\}$}]  (v5) {};	
			
			\draw[thick] (v1)--(v5)--(v2);
			
		\end{tikzpicture}
		\caption{The house graph $\mathcal{H}$ with $i$-sets of $G$.}
		\label{fig:i:houseConstruct:b}%
	\end{subfigure}		
	\caption{The graph $G$ for Proposition \ref{prop:i:house} with $\ig{G}=\mathcal{H}$.}
	\label{fig:i:houseConstruct}%
\end{figure}
 
 The next result shows that maximal cliques in $i$-graphs can be replaced by arbitrarily larger maximal cliques to form larger $i$-graphs.
 
\begin{lemma}[Max Clique Replacement Lemma] \label{lem:i:cliqueRep} \emph{\cite{BMT1, LauraD}}
	Let $H$ be an $i$-graph with a maximal $m$-vertex clique, $\mathcal{K}_m$.  The graph $H_w$ formed by adding a new vertex $w^*$ adjacent to all of $\mathcal{K}_m$ is also an $i$-graph.	
\end{lemma}

The Deletion Lemma below shows that the class of $i$-graphs is closed under vertex deletion, that is, every induced subgraph of an $i$-graph is also an $i$-graph.

\begin{lemma}[Deletion Lemma] \emph{\cite{BMT1, LauraD}}
\label{lem:i:inducedI} 
	If $H$ is a nontrivial $i$-graph, then any induced subgraph of $H$ is also an $i$-graph.
\end{lemma}

\begin{coro} \label{coro:i:notInduced}  \emph{\cite{BMT1, LauraD}}
    If $H$ is not an $i$-graph, then any graph containing an induced copy of $H$
is also not an i-graph.
\end{coro}

When visualizing the connections between the $i$-sets of a graph $G$, it is sometimes advantageous to consider its complement $\cp{G}$ instead. 
From a human perspective, it is curiously easier see to which vertices a vertex $v$ is adjacent, rather than to which vertices $v$ is nonadjacent.
This is especially true when $i(G) = 2$ or $3$, when we may interpret the adjacency of $i$-sets of $G$ as the adjacencies of edges and triangles (i.e. $K_3$), respectively, in~$\cp{G}$.  
In the following sections we examine how the use of graph complements can be exploited to construct the $i$-graph seeds for certain classes of line graphs, theta graphs, and maximal planar graphs.  

\section{Line Graphs and Claw-free Graphs}	\label{sec:c:line}

Consider a graph $G$ with $i(G)=2$ and where $X = \{u,v\}$ is an $i$-set of $G$.  In $\cp{G}$, $u$ and $v$ are adjacent, so $X$ is represented as the edge $uv$.  Moreover, no other vertex $w$ is adjacent to both vertices of $X$ in $\cp{G}$; otherwise, $\{u,v,w\}$ is independent in $G$, contrary to $X$ being an $i$-set.   

Now consider the line graph $L(\cp{G})$ of $\cp{G}$.  If $X=\{u,v\}$ is an $i$-set of $G$, then $e=uv$ is an edge of $\cp{G}$ and hence $e$ is a vertex of $L(\cp{G})$.  Thus, the $i$-sets of $G$ correspond to a subset of the vertices of $L(\cp{G})$. In the case where $\cp{G}$ is triangle-free (that is, $G$ has no independent sets of cardinality $3$), these  $i$-sets of $G$ are exactly the vertices of $L(\cp{G})$.
Now suppose $Y$ is an $i$-set of $G$ adjacent to $X$; say, $\edge{X,u,w,Y}$, so that $Y=\{v,w\}$.  Then, in $\cp{G}$, $f=vw$ is an edge, and so in $L(\cp{G})$, $f \in V(L(\cp{G}))$.  Since $e$ and $f$ are both incident with $v$ in $\cp{G}$, $ef \in E(L(\cp{G}))$.
That is, for $i$-sets $X$ and $Y$ of a well-covered graph $G$ with $i(G) = \alpha(G)=2$, $X \sim Y$ if and only if $X$ and $Y$ correspond to adjacent vertices in $L(\cp{G})$.   Thus $\ig{G} \cong L(\cp{G})$.

In the example illustrated in Figure \ref{fig:c:compEx} below, $\mathcal{H}$ is the house graph, where $X=\{a,c\}$ and $Y=\{c,e\}$ are $i$-sets with $\edge{X,a,e,Y}$.  In $L(\cp{\mathcal{H}})$, the two vertices in each of these $i$-sets are likewise adjacent.

\begin{figure}[H]
	\centering
	\begin{tikzpicture}
		
		\node(cent) at (0,0) {};			
		\path (cent) ++(0:50 mm) node(mcent)  {};	
		\path (cent) ++(0:100 mm) node(rcent)  {};
		
		\foreach \i/\c in {a/-135,b/135,d/45, e/-45} {				
			\path (cent) ++(\c:10 mm) node [std] (v\i) {};
			\path (v\i) ++(\c:5 mm) node (Lv\i) {${\i}$};
		}
		
		\draw[thick] (va)--(vb)--(vd)--(ve)--(va)--cycle;
		
		\path (cent) ++(90:15 mm) node [std] (vc) {};
		\path (vc) ++(90:5 mm) node (Lvc) {${c}$};
		
		\draw[thick] (vb)--(vc)--(vd);	
		
		\path (cent) ++(-90: 20mm) node (HL)  {$\mathcal{H}$};

		\foreach \i/\c in {a/-135,b/135,d/45, e/-45} {				
			\path (mcent) ++(\c:10 mm) node [std] (m\i) {};
			\path (m\i) ++(\c:5 mm) node (Lm\i) {${\i}$};
		}
		
		\path (mcent) ++(90:15 mm) node [std] (mc) {};
		\path (mc) ++(90:5 mm) node (Lmc) {${c}$};
		
		\draw[thick] (mb)--(me)--(mc)--(ma)--(md);
		
		\path (mcent) ++(-90: 20mm) node (HcL)  {$\cp{\mathcal{H}}$};
		
		\foreach \i/\j/\c in {c/e/-135,b/e/135,a/d/45, a/c/-45} {				
			\path (rcent) ++(\c:10 mm) node [std] (r\i\j) {};
			\path (r\i\j) ++(\c:8 mm) node (Lr\i\j) {$\{\i,\j\}$};
		}
		
		\draw[thick] (rbe)--(rce)--(rac)--(rad);
		\path (rcent) ++(-90: 20mm) node (HiL)  {$L(\cp{\mathcal{H}}) \cong \ig{\mathcal{H}}$};
		
	\end{tikzpicture}				
	\caption{The complement and line graphs complement of the well-covered house graph $\mathcal{H}$.}
	\label{fig:c:compEx}
\end{figure}
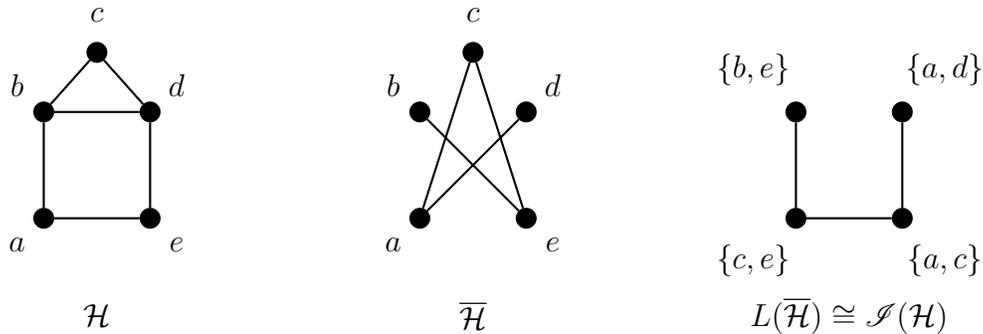





The connection between graphs with $i(G)=2$ and line graphs helps us not only understand the structure of $\ig{G}$, but also lends itself towards some interesting realizability results.  
We follow this thread for the remainder of this section, and build towards determining the $i$-graph realizability of line graphs and claw-free graphs. The straightforward proof of the following lemma can be found in \cite{LauraD} and is omitted here.

\begin{lemma}
	\label{lem:c:lineK3}
	
	The line graph of a connected graph $G$ of order at least $4$ contains $\Dia$ as an induced subgraph if and only if $G$ contains a triangle. (See Figure \ref{fig:c:paw}.)
\end{lemma}

	

\begin{figure}[H] \centering	
	\begin{tikzpicture}	[scale=0.8]		
		\node(cent) at (0,0) {};
		\path (cent) ++(0,0 mm) node(lcent)  {};			
		\path (cent) ++(0:50 mm) node(mcent)  {};	
		
		\node[std,label={0:$a$}]  at (lcent) (va) {}; 
		
		\path (lcent) ++(90:10 mm) node [std] (vd) {};
		\path (vd) ++(90:5 mm) node (Lvd) {${d}$};
		\path (lcent) ++(180:10 mm) coordinate (hc1) {};	
		\path (hc1) ++(-90:10 mm) node [std] (vc) {};
		\path (vc) ++(-90:5 mm) node (Lvc) {${c}$};	
		
		\path (vc) ++(0:20 mm) node [std] (vb) {};
		\path (vb) ++(-90:5 mm) node (Lvb) {${b}$};	
		
		\draw[thick](va)--(vb)--(vc)--(va)--cycle;
		\draw[thick](va)--(vd);	
		
		\path (cent) ++(-90: 22mm) node (HL)  {$H$};		
		
		\foreach \i/\c in {ad/0,ac/90,bc/180, ab/-90} {				
			\path (mcent) ++(\c:10 mm) node [std] (m\i) {};
			\path (m\i) ++(\c:5 mm) node (Lm\i) {${\i}$};
		}
		
		\draw[thick] (mad)--(mac)--(mbc)--(mab)--(mad) -- cycle;
		\draw[thick] (mac)--(mab);
		
		\path (mcent) ++(-90: 22mm) node (HcL)  {$L(H) \cong \Dia$};		
	\end{tikzpicture}
	\caption{The ``paw'' $H$ with $L(H) \cong \Dia$.}
	\label{fig:c:paw}		
\end{figure}


\begin{theorem}
	\label{thm:c:line}  Let $H$ be a connected line graph.  Then $H$ is an $i$-graph if and only if $H$ is $\Dia$-free.	
\end{theorem}

\begin{proof}  Suppose $H$ is an $i$-graph.  By Proposition \ref{prop:i:diamond} 
	and Corollary \ref{coro:i:notInduced},  
	$H$ is $\Dia$-free.
	
	Conversely, suppose that $H$ is $\Dia$-free.  If $H$ is complete, then $H$ is the $i$-graph of itself.  So, assume that $H$ is not complete.  Say $H$ is the line graph of some graph $F$, where we may assume $F$ has no isolated vertices (as isolated vertices do not affect line graphs).  Since $H$ is $\Dia$-free and connected, $F$ has no triangles by Lemma  \ref{lem:c:lineK3}.  Since $F$ has edges (which it does since $H$ exists), $
	\alpha(\cp{F}) \leq 2$.  Moreover, as $F$ is connected, $\cp{F}$ has no universal vertices, and so $i(\cp{F})\geq 2$.   Thus, $i(\cp{F}) = \alpha(\cp{F})$ and $\cp{F}$ is well-covered. It follows that every edge of $F$ corresponds to an $i$-set of $\cp{F}$.  Since $H$ is the line graph of $F$, it is the $i$-graph of $\cp{F}$.  
\end{proof}

\medskip

Finally, if we examine Beineke's forbidden subgraph characterization of line graphs (see \cite{B70} or \cite[Theorem 6.26]{CLZ}), we note that eight of the nine minimal non-line graphs contain an induced $\Dia$ and are therefore not $i$-graphs.  The ninth minimal non-line graph is the claw, $K_{1,3}$.  Thus, $\Dia$-free claw-free graphs are $\Dia$-free line graphs, hence $i$-graphs. 

\begin{coro}	\label{coro:c:clawfree}  
	Let $H$ be a connected claw-free graph.  Then $H$ is an $i$-graph if and only if $H$ is $\Dia$-free.	
\end{coro}


\medskip

While Theorem \ref{thm:c:line} and Corollary \ref{coro:c:clawfree} reveal the $i$-graph realizability of many famous graph families (including another construction for cycles, which are connected, claw-free, and $\Dia$-free), the realizability problem for graphs containing claws remains unresolved.  Moreover, among clawed graphs are the theta graphs which we first alluded to in Section \ref{sec:iGr} as containing three of the small known non-$i$-graphs.  In the next section we apply similar techniques with graph complements to construct all theta graphs that are $i$-graphs.

\section{Theta Graphs From Graph Complements} \label{sec:c:thetaComp}

Consider a graph $G$ with $i(G)=3$.  Each $i$-set of $G$ is represented as a triangle (a $K_3$) in $\cp{G}$.  
If $X$ and $Y$ are two $i$-sets of $G$ with $\edge{X,u,v,Y}$, then $\cp{G}[X]$ and $\cp{G}[Y]$ are triangles in $\cp{G}$, and have $|X \cap Y| =2$.  Although it is technically the induced subgraphs $\cp{G}[X]$ and $\cp{G}[Y]$ that are the triangles of $\cp{G}$,  for notational simplicity we refer to $X$ and $Y$ as triangles.  In $\cp{G}$, the triangle $X$ can be transformed into the triangle $Y$ by removing the vertex $u$ and adding in the vertex $v$  (where $u\not\sim_{\cp{G}} v$).  Thus, we say that two triangles are \emph{adjacent} if they share exactly one edge.  Moreover, since two $i$-sets of a graph $G$ with $i(G)=3$ are adjacent if and only if their associated triangles in $\cp{G}$ are adjacent, we use the same notation for $i$-set adjacency in $G$  as triangle adjacency in $\cp{G}$; that is, the notation $X \sim Y$ represents both $i$-sets $X$ and $Y$ of $G$ being adjacent, and triangles $X$ and $Y$ of $\cp{G}$ being adjacent.

In the following sections we use triangle adjacency to construct \emph{complement seed graphs} for the $i$-graphs of theta graphs; that is, a graph $\cp{G}$ such that $\ig{G}$ is isomorphic to some desired theta graph. Before proceeding with these constructions, we note some observations which will help us with this process.


\begin{obs}\label{obs:c:i3}
	A graph  $G$ has $i(G)=2$ if and only if  $\cp{G}$ is nonempty and has an edge that does not lie on a triangle.
\end{obs}

\noindent If $\cp{G}$ has an edge $uv$ that does not lie on a triangle, then $\{u,v\}$ is independent and dominating in $G$, and so $i(G) \leq 2$.  When building our seed graphs $G$ with $i(G)=3$, it is therefore necessary to ensure that every edge of the complement $\cp{G}$ belongs to a triangle.


\begin{obs}
	Let $G$ be a graph with $i(G)=3$.  If  $S$ with $|S| \geq 4$ is a (possibly non-maximal) clique in $\cp{G}$, then no $3$-subset of $S$ is  an $i$-set of $G$.
\end{obs}

\noindent Suppose that $S=\{u,v,w,x\}$ is such a clique of $\cp{G}$.  Then, for example, $x$ is undominated by $\{u,v,w\}$ in $G$, and so $\{u,v,w\}$ is not an $i$-set of $G$.  
Conversely, suppose that $\{u,v,w\}$ is a triangle in a graph $\cp{G}$ with $i(G)=3$.  By attaching a new vertex $x$ to all of $\{u,v,w\}$ in $\cp{G}$, we remove $\{u,v,w\}$ as an $i$-set of $G$, while keeping all other $i$-sets of $G$.  This observation proves to be very useful in the constructions  in the following section: we now have a technique to eliminate any unwanted triangles in $\cp{G}$ (and hence $i$-sets of $G$) that may arise.  Note that this technique is an application of the Deletion Lemma (Lemma \ref{lem:i:inducedI}) in $\cp{G}$ instead of in~$G$.


The use of triangle adjacency in a graph $\cp{G}$ to determine $i$-set adjacency in $G$ provides a key technique to resolve the question of
which theta graphs are $i$-graphs.  In our main result, Theorem \ref{thm:c:thetas}, we show that all theta graphs except the seven listed exceptions are $i$-graphs.  Using this method of complement triangles, the proofs of the lemmas for the affirmative cases make up most of the remainder of Section \ref{sec:c:thetaComp}.  The proofs of  the lemmas for the seven negative cases are given in Section \ref{sec:c:nonthetas}.  

\begin{theorem}\label{thm:c:thetas}
	
	A theta graph is an $i$-graph if and only if it is not one of the seven exceptions listed below: 
	
	\begin{tabular}{lllll}  \centering
		$\thet{1,2,2}$, 	&&&&	\\
		$\thet{2,2,2}$,	\	&		$\thet{2,2,3}$,	\	&		$\thet{2,2,4}$,  \  &	$\thet{2,3,3}$, \ 	&		$\thet{2,3,4}$, \ 	\\
		$\thet{3,3,3}.$ &&&& \\
	\end{tabular}

\end{theorem}


\noindent Table \ref{tab:c:thetaLong} summarizes the cases used to establish Theorem \ref{thm:c:thetas} and their associated results.

\begin{table}[H] \centering
	\begin{tabular}{|l|l|l|}  \hline
		$\thet{j,k,\ell}$	&	Realizability	&	Result	\\	\hline
		$\thet{1,2,2}$	&	non-$i$-graph	&	$\Dia$.  Proposition \ref {prop:i:diamond}	\\	
		$\thet{1,2,\ell}$, $\ell \geq 3$ 	&	$i$-graph	&	Lemma \ref{lem:c:thet12k}	\\	
		$\thet{1,k,\ell}$, $3 \leq k \leq \ell$ 	&	$i$-graph	&	Lemma \ref{lem:c:thet1kl}	\\	
		&		&		\\	
		$\thet{2,2,2}$	&	non-$i$-graph	&	$K_{2,3}$.  Proposition \ref{prop:i:diamond}	\\	
		$\thet{2,2,3}$	&	non-$i$-graph	&	$\kappa$.  Proposition \ref{prop:i:diamond}	\\	
		$\thet{2,2,4}$	&	non-$i$-graph	&	Proposition \ref{prop:c:thet224}	\\	
		$\thet{2,2,\ell}$, $\ell\geq 5$	&	$i$-graph	&	Lemma \ref{lem:c:thet22l}	\\	
		$\thet{2,3,3}$	&	non-$i$-graph	&	Proposition \ref{prop:c:thet233}	\\	
		$\thet{2,3,4}$	&	non-$i$-graph	&	Proposition \ref{prop:c:thet234}	\\	
		$\thet{2,3,\ell}$, $\ell\geq 5 $	&	$i$-graph	&	Lemma \ref{lem:c:thet23k}	\\	
		$\thet{2,4,4}$	&	$i$-graph	&	Lemma \ref {lem:c:thet244}	\\	
		$\thet{2,k,5}$, $4 \leq k \leq 5$	&	$i$-graph	&	Lemma \ref{lem:c:thet2k5}	\\	
		$\thet{2,k,\ell}$, $k \geq 4$, $\ell \geq 6$, $\ell \geq k$	&	$i$-graph	&	 Lemma \ref{lem:c:thet2kl}	\\	
		&		&		\\	
		$\thet{3,3,3}$	&	non-$i$-graph	&	 Proposition \ref{prop:c:thet333}	\\	
		$\thet{3,3,4}$	&	$i$-graph	&	Lemma \ref{lem:c:thet334}	\\	
		$\thet{3,3,5}$	&	$i$-graph	&	Lemma \ref {lem:c:thet335}	\\	
		$\thet{3,3,\ell}$, $\ell\geq 6$	&	$i$-graph	&	Lemma \ref {lem:c:thet33k}	\\	
		$\thet{3,4,4}$	&	$i$-graph	&	Lemma \ref{lem:c:thet344}	\\	
		$\thet{3,4,\ell}$, $\ell\geq 5$	&	$i$-graph	&	Lemma \ref{lem:c:thet34k}	\\	
		$\thet{3,5,5}$	&	$i$-graph	&	Lemma \ref{lem:c:thet355}	\\	
		&		&		\\	
		$\thet{4,4,4}$	&	$i$-graph	&	Lemma \ref{lem:c:thet444}	\\	
		$\thet{j,k,5}$, $4 \leq j \leq k \leq 5 $.	&	$i$-graph	&	Lemma \ref{lem:c:thetjk5}	\\	
		$\thet{j,k,\ell}$, $3 \leq j \leq k \leq \ell$, and $\ell \geq 6$.	&	$i$-graph	&	Lemma \ref{lem:c:thetjkl}	\\	\hline
	\end{tabular} \caption{$i$-graph realizability of theta graphs.} \label{tab:c:thetaLong}
\end{table}

\subsection{$\thet{1,k,\ell}$}  \label{subsec:c:1kl}

We have already seen that the house graph $\mathcal{H} = \thet{1,2,3}$ is an $i$-graph (Proposition \ref{prop:i:house}).  We can further exploit previous results to see that all graphs $\thet{1,2,\ell}$ for $\ell \geq 3$ are $i$-graphs by taking a cycle $C_n$ with $n \geq 4$, and replacing one of its maximal cliques (i.e. an edge) with a $K_3$.  By the Max Clique Replacement Lemma (Lemma \ref{lem:i:cliqueRep}), the resultant $\thet{1,2,n-1}$ is also an $i$-graph.  For reference, we explicitly state this result as a lemma.

\begin{lemma} \label{lem:c:thet12k}  
	For $\ell \geq 3$, the theta graph $\thet{1,2,\ell}$ is an $i$-graph.
\end{lemma}


\subsubsection{Construction of $\cp{G}$ for $\thet{1,k,\ell}$, $3 \leq k \leq \ell$}  \label{subsubsec:c:1kl+}

In Figure \ref{fig:c:thet145} below, we provide a first example of the technique we employ repeatedly throughout this section to construct our theta graphs.  To the left is a graph $\cp{G}$, where each of its nine triangles corresponds to an $i$-set of its complement $G$.  The resultant $i$-graph of $G$, $\ig{G} = \thet{1,4,5}$, is presented on the right.  For consistency, we use $X$ and $Y$ to denote the triangles corresponding to the degree $3$ vertices in the theta graphs in this example, as well as all constructions to follow.

\begin{figure}[H] \centering	
	\begin{tikzpicture}	[scale=0.8]
		
		\node (cent) at (0,0) {};
		\path (cent) ++(0:40 mm) node (rcent) {};	
		
		\node [std] at (cent) (w0) {};
		\path (w0) ++(40:4 mm) node (w0L) {$w_0$};
		
		\foreach \i/\c in {1/90,3/270,4/210, 5/150} {				
			\path (w0) ++(\c:20 mm) node [std] (w\i) {};
			\path (w\i) ++(\c:5 mm) node (Lw\i) {$w_{\i}$};
			\draw [thick] (w0)--(w\i);
		}
		
		\path (w0) ++(0:20 mm) node [std] (w2) {};
		\path (w2) ++(90:5 mm) node (Lw2) {$w_{2}$};
		\draw [thick] (w0)--(w2);	
		
		\draw  [thick] (w1)--(w2)--(w3)--(w4)--(w5)--(w1);		
		
		\foreach \i/\c in {1/15,2/0,3/-15} {				
			\path (rcent) ++(90:\c mm) node [bblue] (v\i) {};
			\path (v\i) ++(0:5 mm) node [color=blue] (Lv\i) {$v_{\i}$};
			\draw[bline] (w2)--(v\i);
		}
		\draw[bline] (w1)--(v1)--(v2)--(v3)--(w3);	
		
		\path (w0) ++(45:10 mm) node  [color=red] (LX) {$X$};
		\path (w0) ++(-45:10 mm) node  [color=red] (LY) {$Y$};
		
		\foreach \i/\c in {1/120,2/180,3/240} {				
			\path (w0) ++(\c:12 mm) node [color=red] (A\i) {$A_{\i}$};	
		}		
		
		\foreach \i/\c in {1/90, 2/18, 3/-18, 4/-90} {				
			\path (w2) ++(\c:12 mm) node [color=red] (B\i) {$B_{\i}$};	
		}	
		\node[broadArrow,fill=white, draw=black, ultra thick, text width=6mm] at (5.8cm,0){};
	\end{tikzpicture}
	\begin{tikzpicture}		
		\node (cent) at (0,0) {};
		\node  at (cent) (w0) {};
		\path (w0) ++(0:20 mm) node (w2) {};	
		
		\path (w0) ++(45:10 mm) node [bred] (X) {};
		\path (X) ++(90:5 mm) node [color=red] (LX) {$X$};
		
		\path (w0) ++(-45:10 mm) node [bred] (Y) {};
		\path (Y) ++(-90:5 mm) node [color=red] (LY) {$Y$};
		
		\draw [rline] (X)--(Y);	
		
		\foreach \i/\c in {1/120,2/180,3/240} {			
			\path (w0) ++(\c:12 mm) node [bred] (A\i) {};		
			\path (A\i) ++(\c:5 mm) node [color=red] (LA\i) {$A_{\i}$};	
		}		
		\draw [rline] (X)--(A1)--(A2)--(A3)--(Y);
		
		\foreach \i/\c in {1/90, 2/22, 3/-22, 4/-90} {				
			\path (w2) ++(\c:15 mm) node [bred] (B\i) {};
			\path (B\i) ++(\c:5 mm) node [color=red] (LB\i) {$B_{\i}$};	
		}	
		\draw [rline] (X)--(B1)--(B2)--(B3)--(B4)--(Y);	
		
		\path (Y) ++(-90: 20mm) node (S) {};		
		
	\end{tikzpicture}
	\caption{A graph $\cp{G}$ (left) such that $\ig{G} = \thet{1,4,5}$ (right).}
	\label{fig:c:thet145}	
\end{figure}						

We proceed now to the general construction of a graph $\cp{G}$ with $\ig{G}=\thet{1,k,\ell}$ for $3 \leq k \leq \ell$.   As it is our first construction using this triangle technique, we provide the construction and proof for Lemma \ref{lem:c:thet1kl} with an abundance of detail.

\begin{cons} \label{cons:c:thet1kl} (See Figure \ref{fig:c:thet1kl}.)  
	Let $\overline{H}\cong W_{k+2}=C_{k+1}\vee K_{1}$, where
	$C_{k+1}=(w_{1},...,
 \newline
 w_{k+1},w_{1})$ and $w_{0}$ 
	is the central hub, i.e., the vertex with degree $k+1$.  Add a path $P_{\ell-2}%
	:(v_{1},...,v_{\ell-2})$, joining each $v_{i},\ i=1,...,\ell-2$, to $w_{2}$.
	Join $v_{1}$ to $w_{1}$ and $v_{\ell-2}$ to $w_{3}$. (If $\ell=3$, then
	$v_{1}=v_{\ell-2}$, hence $v_{1}$ is adjacent to $w_{1},w_{2}$ and $w_{3}$.)
	This is the (planar) graph $\overline{G}$.
\end{cons}

\begin{lemma} \label {lem:c:thet1kl}	
	If $\cp{G}$ is the graph constructed by Construction \ref{cons:c:thet1kl}, then $\ig{G} = \ag{G} = \thet{1,k,\ell}$, $ 3 \leq k \leq \ell$.
\end{lemma}

\begin{figure}[H] \centering	
	\begin{tikzpicture}	[scale=0.8]
		
		\node (cent) at (0,0) {};
		\path (cent) ++(0:45 mm) node (rcent) {};	
		
		\node [std] at (cent) (w0) {};
		\path (w0) ++(15:6 mm) node (w0L) {$w_0$};
		
			\foreach \i/\c /\l in {1/45/1,3/-45/3,4/-90/4, 5/-135/{5}, 6/135/{k}, 7/90/{k+1}} {						
				\path (w0) ++(\c:20mm) node [std] (w\i) {};
					\path (w\i) ++(\c:6 mm) node (Lw\i) {$w_{\l}$};
				\draw[thick] (w0)--(w\i);
			}
			
			\path (w0) ++(0:20 mm) node [std] (w2) {};
				\path (w2) ++(80:5 mm) node (Lw2) {$w_{2}$};
			\draw[thick]  (w0)--(w2);	
			
			\draw[thick]  (w1)--(w2)--(w3)--(w4)--(w5);
			\draw[thick]  (w6)--(w7)--(w1);		
			
			\path (w0) ++(165:20mm) coordinate (z1) {};
			\path (w0) ++(195:20mm)coordinate (z2) {};
			
			\draw[thick] (w5)--(z2)--(z1)--(w6);
			
			\draw [dashed,thick] (z2)--(w0);
			\draw [dashed,thick] (z1)--(w0);
			
			\foreach \i / \j  / \k in {1/2.2/1, 2/1/2, 3/-1/{\ell-3}, 4/-2.2/{\ell-2}}
			{	
				\path (rcent) ++(90:\j cm) node [bblue] (v\i) {};
				\path (v\i) ++(0:7 mm) node (Lv\i) {\color{blue} $v_{\k}$};
				
				\draw[bline] (w2)--(v\i);
			}
			
			\draw[bline] (v1)--(v2);
			\draw[bline] (v3)--(v4);
			
			\draw[bline] (v2)--(v3);
			
			\draw[bline] (v1)--(w1);
			\draw[bline] (v4)--(w3);	
			
			\coordinate  (q1) at ($(v2)!0.33!(v3)$) {};
			\coordinate  (q2) at ($(v2)!0.66!(v3)$) {};
			
			\draw [dashed,bline] (q2)--(w2);
			\draw [dashed,bline] (q1)--(w2);		
			
			\path (w0) ++(22:14 mm) node  [color=red] (LX) {$X$};
			\path (w0) ++(-22:14 mm) node  [color=red] (LY) {$Y$};		
			
			\foreach \i/\c / \l in {1/67/1,2/112/2,3/-115/{k-2}, 4/-67/{k-1}} 			
			\path (w0) ++(\c:14 mm) node [scale=0.9, color=red] (A\i) {${A_{\l}}$};				
			
			\foreach \i/\c /\l in {1/80/1,  4/-80/{\ell-1}} 				
			\path (w2) ++(\c:11 mm) node [scale=0.9,color=red] (B\i) {$B_{\l}$};	
			\foreach \i/\c /\l in { 2/32/2, 3/-30/{\ell-2}} 			
			\path (w2) ++(\c:20 mm) node [scale=0.9,color=red] (B\i) {$B_{\l}$};	
			
		\end{tikzpicture} \hspace{5mm}	
		\caption{The graph $\cp{G}$ from Construction \ref{cons:c:thet1kl} such that $\ig{G} = \thet{1,k,\ell}$ for $3 \leq k \leq \ell$.}
		\label{fig:c:thet1kl}	
	\end{figure}						

	\begin{proof}
		To begin, notice that since in $\cp{G}$, the vertices $\mathcal{W}=\{w_0,w_1,\dots,w_{k},w_{k+1}\}$ form a wheel on at least five vertices, $\cp{H} \ncong K_4$.  Likewise, the graph induced by $\{w_0, w_1, w_2, w_3, v_1, \newline v_2,$ $\dots, v_{\ell-2}\}$ in $\cp{G}$ is also a wheel on $\ell +2$ vertices, where $w_2$ is the central hub, and so it too contains no $K_4$.  Therefore, $\cp{G}$ is $K_4$-free and the triangles of $\cp{G}$ are precisely the maximal cliques of $\cp{G}$, and so $\omega(\cp{G}) = i(G) = \alpha(G) = 3$.  Since the $i$-sets of $G$ are identical to its $\alpha$-sets, $\ig{G}=\ag{G}$, and so for ease of notation, we will refer only to  $\ig{G}$ throughout the remainder of this proof.
		
		We label the triangles as in Figure \ref{fig:c:thet1kl} by dividing them into two collections.  The first are the triangles composed only of the vertices from $\mathcal{W}$ and each containing $w_0$: let $X = \{w_0,w_1,w_2\}$, $Y = \{w_0, w_2,w_3\}$, $A_1 =\{w_0,w_{k+1},w_1\} $, $A_2=\{w_0, w_{k}, w_{k+1}\}$, $\dots$ $A_{k-2} = \{w_0, w_{4}, w_{5}\}$,  $A_{k-1} = \{w_0, w_3, w_4\}$.
		The remainder are the triangles with vertex sets not fully contained in $\cp{H}$: $B_1 = \{w_2,w_1,v_1\}$, $B_2 = \{w_2, v_1, v_2\}$, $B_3 = \{w_2, v_2, v_3\}$, $\dots$, $B_{\ell-2} = \{w_2, v_{\ell-3}, v_{\ell-2}\}$,  and $B_{\ell-1} = \{w_2, v_{\ell-2},w_3\}$.  We refer to these collections as $\mathcal{S} = \{X, Y\}$,  $\mathcal{A}=\{A_1,A_2,\dots,A_{k-1}\}$, and $\mathcal{B} = \{B_1, B_2, \dots, B_{\ell-1}\}$. It is clear from Figure \ref{fig:c:thet1kl} that these are the only triangles of $\cp{G}$.  Therefore $V(\ig{G})=\{X,Y, A_1,A_2,\dots,A_{k-1}, B_1,B_2,\dots,B_{\ell-1}\}$.

		We now show that the required adjacencies hold.  		
		From the construction of $\cp{G}$, the following are immediate for $\ig{G}$:
		
		\begin{enumerate}[itemsep=0pt, label=(\roman*)]
			\item \label{it:c:thet1kl:1}  $\edge{X, w_1, w_3,Y}$, 
			
			\item \label{it:c:thet1kl:j} $\edge{X, w_0,  v_1,B_1}  \adedge{w_1,v_2,B_2} \adedge{v_1,v_3,B_3} \dots \edge{B_{\ell-2}, v_{\ell-3},  w_3,B_{\ell-1}} \adedge{v_{\ell-2},w_0,Y}$, 
			
			\item \label{it:c:thet1kl:k}  $\edge{X, w_2,  w_{k+1}, A_1}  \adedge{w_1,w_k,A_2} \dots \edge{A_{k-2}, w_{5},  w_3, A_{k-1}} \adedge{w_{4},w_2,Y}$. 
		\end{enumerate}
		
		\noindent Hence, we need only show that there are no additional unwanted edges generated in the construction of $\ig{G}$.  	
		
		Since $\cp{G}$ is a planar graph and all of its triangles are facial (that is, the edges of the $K_3$ form a face in the plane embedding in Figure \ref{fig:c:thet1kl}), each triangle is adjacent to at most three others.  From \ref{it:c:thet1kl:1} -- \ref{it:c:thet1kl:k}  above, triangles $X$ and $Y$ are both adjacent to the maximum three (and hence $\deg_{\ig{G}}(X)=\deg_{\ig{G}}(Y)=3$).  
		
		Recall that to be adjacent, two triangles share exactly two vertices.  Notice that the triangles of $\mathcal{A}$ are composed entirely of vertices from $\mathcal{W}-\{w_2\}$, and that for $2\leq i \leq \ell-2$, $B_i\cap\mathcal{W}=\{w_2\}$; furthermore, $B_1\cap\mathcal{W}=\{w_1,w_2\}$ and $B_{\ell-2}\cap\mathcal{W}=\{w_2,w_3\}$. Therefore no triangle of $\mathcal{B}$ is adjacent to any triangle of $\mathcal{A}$. It is similarly easy to see that there are no additional unwanted adjacencies between two triangles of $\mathcal{A}$ or two triangles of $\mathcal{B}$.
		
		
  
  We conclude that the graph $\cp{G}$ generated by Construction \ref{cons:c:thet1kl} yields $\ig{G} =\ag{G}= \thet{1,k,\ell}$.
	\end{proof}
	
 
 \medskip
	
	Before we proceed with the remainder of the theta graph constructions, let us return to Figure \ref{fig:c:thet1kl} to notice the prominence of the wheel subgraph in the complement seed graph $\cp{G}$.  In the constructions throughout this paper, this wheel subgraph will appear repeatedly; indeed, all of the complement seed graphs for the $i$-graphs of theta graphs have a similar basic form: begin with a wheel, add a path of some length, and then add some collection of edges between them.   As stated without proof in Lemma \ref{lem:c:wheel}, Figure \ref{fig:c:wheel} below demonstrates that a wheel in a triangle-based complement seed graph $\cp{G}$ corresponds to a cycle in the $i$-graph of $G$.  Using this result, in our later constructions with a wheel subgraph, we already have two of the three paths of a theta graph formed. Hence, we need only confirm that whatever unique additions are present in a given construction form the third path in the $i$-graph.

\begin{figure}[H] \centering	
	\begin{tikzpicture}	[scale=0.8]
		\node (cent) at (0,0) {};
		
		\node[std,label={0:$w_0$}] (w0) at (cent) {};
		
		\foreach \i/\j in {1/1,2/2,3/3,4/4,5/5,6/k}
		{
			\node[std,label={150-\i*60: $w_{\j}$}] (w\i) at (150-\i*60:2.5cm) {}; 
			\draw[thick] (w0)--(w\i);        
		}
		
		\draw[thick] (w1)--(w2)--(w3)--(w4)--(w5);
		\draw[thick] (w6)--(w1);
		\draw[thick] (w5)--(w6);
		
		\coordinate  (v1) at ($(w5)!0.33!(w6)$) {};
		\coordinate  (v2) at ($(w5)!0.66!(w6)$) {};		
		
		\draw [dashed, thick] (v2)--(w0);
		\draw [dashed, thick] (v1)--(w0);				
		
		\foreach \i/\j in {1/1,2/2,3/3,4/4,6/k}
	{	\node[dred] (a\i) at (120-\i*60:12mm) {}; 
			\path (a\i) ++(120-\i*60:6 mm) node (Aw\i) {\color{red}$A_{\j}$};
	}
	
	
		\draw[thick,red, densely dashdotted](a6)--(a1)--(a2)--(a3)--(a4);
		
		\foreach \i/\j in {5/195,6/180,7/165}
		\coordinate (b\i) at (\j:12mm) {}; 		
		
		\draw[dashed,red, thick] (a4)--(b5)--(b6)--(b7)--(a6);
		
	\end{tikzpicture} 
	\caption{The wheel  $H_k=W_{k+1}$, with the $i$-graph of its complement, $\ig{\cp{H_k}} \cong \ag{\cp{H_k}} \cong  C_k$, embedded in red.}
	\label{fig:c:wheel}
\end{figure}		

	\begin{lemma}\label{lem:c:wheel}
		For $k\geq 4$, let $H_{k}$ be the wheel $W_{k+1} =C_{k} \vee K_1$.  Then $\ig{\cp{H_{k}}} \cong \ag{\cp{H_{k}}} \cong C_k$.
	\end{lemma}

	We note the following analogous  - although less frequently applied - result for fans of the form $K_1+P_k$, as illustrated in Figure \ref{fig:c:fan}.


\begin{figure}[H] \centering	
	\begin{tikzpicture}	[scale=0.8]
		\node (cent) at (0,0) {};
		\path (cent) ++(0:40 mm) node (rcent) {};	
		
		\node[std,label={180:$v_0$}] (w0) at (cent) {};
		
		\foreach \i/\d/\j in {1/1/1,2/2/2,3/3/3,4/4/4,5/6/{k-1},6/7/{k}}
		{
			\path (rcent) ++(90:40mm-\d*10mm) node[std,label={0: $v_{\j}$}] (w\i) {};
			\draw[thick] (w0)--(w\i);        
		}
		
		\draw[thick] (w1)--(w2)--(w3)--(w4);
		\draw[thick] (w5)--(w6);
		\draw[dashed, thick] (w4)--(w5);
		
		\coordinate  (v1) at ($(w4)!0.33!(w5)$) {};
		\coordinate  (v2) at ($(w4)!0.66!(w5)$) {};
		
		\draw [thick, dashed] (v2)--(w0);
		\draw [thick, dashed] (v1)--(w0);	
		
		\foreach \i/\d /\r /\j in {1/1/25/1,2/2/10/2,3/3/0/3,5/6/-25/{k-1}}
		{
			\path (rcent) ++(90:35mm-\d*10mm) coordinate  (z\i) {};	
			\coordinate  (y\i) at ($(z\i)!0.35!(w0)$) {};
			\node[dred] (a\i) at (y\i) {};	
				\path (a\i) ++(\r:6mm) node[scale=0.9]  (AL\i) {\color{red}$A_{\j}$};	
		}
		
		\path (a3) ++(270:5mm) coordinate (d3) {};	
		\path (a5) ++(90:5mm) coordinate (d5) {};	
		
		\draw[thick,red](a1)--(a2)--(a3);
		\draw[thick,red](a3)--(d3);
		\draw[thick,red](a5)--(d5);
		
		\draw[dashed,thick, red, densely dashdotted] (d3)--(d5);
	\end{tikzpicture} 	
	\caption{The  fan  $H_k=K_1 \vee P_k$, with the $i$-graph of its complement, $\ig{\cp{H_k}} \cong \ag{\cp{H_k}} \cong P_{k-1}$, embedded in red.}
	\label{fig:c:fan}	
\end{figure}

	\begin{lemma}\label{lem:c:fan}
		For $k\geq 2$, let $H_{k}$ be the $k$-fan $K_1 \vee P_{k}$.  Then $\ig{\cp{H_{k}}} \cong \ag{\cp{H_k}} \cong P_{k-1}$.
	\end{lemma}

	\subsection{$\thet{2,k,\ell}$  for $2  \leq k  \leq \ell$ }  \label{subsec:c:2kl}

	\subsubsection{Construction of $\cp{G}$ for $\thet{2,2,\ell}$, $\ell \geq 5$}  \label{subsubsec:c:22l}

	As stated in Proposition \ref{prop:i:diamond},  $K_{2,3} \cong \thet{2,2,2}$  and $\kappa \cong \thet{2,2,3}$  are not $i$-graphs.   Extending these results, we find that the length of the third path in $\thet{2,2,\ell}$ has a transition point between $\ell=4$ and $\ell=5$; while $\ell=4$ is still too short to form an $i$-graph (see Lemma \ref{prop:c:thet224}), for  $\ell \geq 5$, $\thet{2,2,\ell}$ is $i$-graph realizable.


	\begin{cons} \label{con:c:thet22k}
		See Figures \ref{fig:c:thet22k} and \ref{fig:c:thet225}.	Begin with the graph $\cp{H} \cong W_5 = C_4 \vee K_1$, labelling the degree $3$ vertices as $w_1,w_2,w_3,w_4$ and the central degree $4$ vertex as~$w_0$.

		\begin{enumerate}[itemsep=0pt, label=(\alph*)]			
  
  \item
  \label{con:c:thet22k:ell}
			If $\ell \geq 6$ (as in Figure \ref{fig:c:thet22k}), attach to $\cp{H}$ a path  $P_{\ell -3}: (v_1, v_2,\dots,v_{\ell-3})$ by joining $w_1$ to $v_1, v_2, \dots, v_{\ell-4}$.  Join $v_{\ell-5}$ to $v_{\ell-3}$.  Next, join $w_2$ to $v_1$, and $w_3$ to $v_{\ell-3}$.  Then, join $w_4$ to $v_{\ell-4}$ and $v_{\ell-3}$.  Add a new vertex $z$, joined to $w_1, w_4,$ and $v_{\ell-4}$.

  \item 
     \label{con:c:thet22k:5} 
            If $\ell = 5$ (as in Figure \ref{fig:c:thet225}), attach to $\cp{H}$ a path of $P_{2}: (v_1, v_2)$, by joining $w_1$ to $v_1$, and $w_2$ to $v_1$ and $v_2$.  Then join $w_3$ to $v_2$, and $w_4$ to both $v_1$ and $v_2$.  Add two new vertices, $z_1$ and $z_2$, joining $z_1$ to $v_1, w_1$, and $w_4$, and $z_2$ to $v_2, w_2$ and $w_3$.
			
		\end{enumerate}
		
		We label the resultant (planar) graph $\cp{G}$.	
	\end{cons}

\begin{figure}[H] \centering	
	\begin{tikzpicture}	[scale=0.8]
		\node (cent) at (0,0) {};
		\path (cent) ++(0:45 mm) node (rcent) {};	
		
		\node[std] (w0) at (cent) {};
			\path (cent) ++(20:5 mm) node (Lw0) {$w_{0}$};
		
		\foreach \i in {1,2,4}
		{
			\node[std,label={180-\i*90: $w_{\i}$}] (w\i) at (180-\i*90:2cm) {}; 
			\draw[thick] (w0)--(w\i);        
		}
	
		\node[std,label={0: $w_{3}$}] (w3) at (180-270:2cm) {}; 
		\draw[thick] (w0)--(w3);

		\draw[thick] (w1)--(w2)--(w3)--(w4)--(w1);
		
		\foreach \i / \j  / \k in {1/3/1, 2/2/2, 3/-1/{\ell-5},  5/-3.4/{\ell-3}}
		{	
			\path (rcent) ++(90:\j cm) node [bblue] (v\i) {};
		}
	
			\path (rcent) ++(90:-2.2cm) node [bblue] (v4) {};
		
		\path (v1) ++(0:4 mm) node (Lv1) {\color{blue} $v_{1}$};
		\path (v2) ++(180:4 mm) node (Lv2) {\color{blue} $v_{2}$};
		\path (v3) ++(170:8 mm) node (Lv3) {\color{blue} $v_{\ell-5}$};
		\path (v4) ++(15:10 mm) node (Lv4) {\color{blue} $v_{\ell-4}$};
		\path (v5) ++(260:4 mm) node (Lv5) {\color{blue} $v_{\ell-3}$};		
		
		\draw[bline] (v1)--(v2);
		\draw[bline] (v3)--(v4)--(v5);
		
		\draw[bline,dashed] (v2)--(v3);
		
		\path (v4) ++(0:17 mm) node [bzed, label={[zelim2]0:$\boldsymbol{z}$}] (z) {};
		
		\path (v4) ++(270:2.6cm) coordinate  (c1) {};
		\draw[zline] (z) to[out=250, in=20]  (c1) to [out=200,in=250] (w4); 
		
		\path (v1) ++(90:2cm) coordinate  (c2) {};
		
		\path (v4) ++(180:1.5 cm) coordinate  (c2) {};
		\draw[bline] (v3) to[out=190, in=90]  (c2) to [out=270,in=170] (v5);

		\draw[zline] (w1) to[out=50, in=170]  (35:6.9cm) to [out=350,in=80]  (z); 
		
		\draw[zline] (v4)--(z);		
		
		
		\draw[bline] (w1)--(v1)--(w2);	
		
		\draw[bline] (v5) to[out=180,in=-90] (w4);  
		
		\draw[bline] (v5) to[out=170,in=-30] (w3);  
		
		\draw[bline] (v4) to[out=-45, in=25]  (4.6cm,-4.2cm) to [out=205,in=260] (w4);  	
		
		\draw[bline] (w1) to[out=20, in=170]  (35:6.1cm) to [out=350,in=45]  (v2); 
		
		\draw[bline] (w1) to[out=38, in=170]  (35:6.4cm) to [out=350,in=45]  (v3); 
		
		\draw[bline] (w1) to[out=46, in=170]  (35:6.6cm) to [out=350,in=40]  (v4); 
		
		\path (rcent) ++(90:1 cm) coordinate  (t1) {};
		\draw[bline, dashed] (w1) to[out=28, in=170]   (35:6.2cm)  to [out=350,in=45]  (t1); 
		
		\path (rcent) ++(90:0 cm) coordinate  (t2) {};
		\draw[bline, dashed] (w1) to[out=34, in=170]  (35:6.3cm) to [out=350,in=45]  (t2); 
		
		\path (w0) ++(45:9 mm) node [color=red] (X) {$X$};			
		\path (w0) ++(45:-9 mm) node [color=red] (Y) {$Y$};
		
		\path (w0) ++(-45:8 mm) node [color=red] (A1) {$A$};	
		
		\path (w0) ++(135:8 mm) node [color=red] (B1) {$B$};			
		
		\path (w2) ++(90:15 mm) node [color=red] (D1) {$D_1$};	
		\path (v1) ++(170:9 mm) node [color=red] (D2) {$D_2$};	
		\path (v4) ++(180:7 mm) node [color=red] (D3) {$D_{\ell-3}$};
		\path (v5) ++(195:20mm) node [color=red] (D4) {$D_{\ell-2}$};
		\path (w4) ++(-56:21 mm) node [color=red] (D5) {$D_{\ell-1}$};			
		
	\end{tikzpicture} 
	\caption{The graph $\cp{G}$ from Construction \ref{con:c:thet22k} such that $\ig{G} = \thet{2,2,\ell}$ for $\ell \geq 6$.}
	\label{fig:c:thet22k}
\end{figure}	

\begin{figure}[H] \centering	
	\begin{tikzpicture}	[scale=0.8]
		
		\node (cent) at (0,0) {};
		\path (cent) ++(0:45 mm) node (rcent) {};	
		
		\node[std] (w0) at (cent) {};
		\path (w0) ++(30: 4 mm) node (Lw0) {${w_0}$};	
		
		\foreach \i in {2,4}
		{
			\node[std,label={180-\i*90: $w_{\i}$}] (w\i) at (180-\i*90:2.5cm) {}; 
			\draw[thick] (w0)--(w\i);        
		}
	
		\node[std,label={180: $w_{1}$}] (w1) at (180-90:2.5cm) {}; 
			\draw[thick] (w0)--(w1);   
			
		\node[std,label={300: $w_{3}$}] (w3) at (180-270:2.5cm) {}; 
			\draw[thick] (w0)--(w3);

		\draw[thick] (w1)--(w2)--(w3)--(w4)--(w1);
		
		\foreach \i / \j in {1/1.5,2/-1.5}
		{	
			\path (rcent) ++(90:\j cm) node [std] (v\i) {};
		}
		
			\path (v1) ++(0:5 mm) node (Lv1) {$v_{1}$};
			\path (v2) ++(290:5 mm) node (Lv2) {$v_{2}$};	
	
		\draw[thick] (w1)--(v1)--(v2)--(w3);
		\draw[thick] (v2)--(w2)--(v1);	
		
		\path (w2) ++(270:1 cm) node [bzed] (z2) {};
			\path (z2) ++(270:4 mm) node (Lz2) {\color{zelim} $\boldsymbol{z_{2}}$};
		\draw[zline] (w2)--(z2)--(w3);
		\draw[zline] (z2)--(v2);
		
		\path (w1) ++(125:1.2 cm) node [bzed] (z1) {};
			\path (z1) ++(5:6 mm) node (Lz1) {\color{zelim} $\boldsymbol{z_{1}}$};
		\draw[zline] (w4)--(z1)--(w1);
		\draw[zline] (z1)--(v1);
		
		\draw[thick] (v2) to[out=230, in=-5]  (270:3.3cm) to [out=170,in=280] (w4);
		
		\draw[thick] (v1) to[out=130, in=0]  (90:4.3cm) to [out=180,in=90] (w4);	
		
		
		\foreach \i / \j / \rad in {1/X/45,2/B/135,3/Y/200,4/A/-20}
		{	
			\node[dred] (t\i) at (-45+\i*90:1.0cm) {}; 
			\path (t\i) ++(\rad: 4 mm) node (Lt\i) {\color{red}${\j}$};				
		}
		

		\path (w2) ++(90: 7mm) node [dred] (d1) {};
		\path (d1) ++(90: 4 mm) node (Ld1) {\color{red}${D_1}$};
		
		\path (w2) ++(0: 15mm) node [dred] (d2) {};
		\path (d2) ++(90: 4 mm) node (Ld2) {\color{red}${D_2}$};	
		
		\path (v2) ++(-55: 15mm) node [dred] (d3) {};
		\path (d3) ++(0: 5 mm) node (Ld3) {\color{red}${D_3}$};	
		
		\path (t3) ++(260: 15mm) node [dred] (d4) {};
		\path (d4) ++(130: 6mm) node (Ld4) {\color{red}${D_4}$};					
		
		\draw[rlinemb](t1)--(t2)--(t3)--(t4)--(t1);
		\draw[rlinemb](t1)--(d1)--(d2);	
		\draw[rlinemb](t3)--(d4);
		
		\draw[rlinemb] (d2) to [out=-45,in=90] (d3);
		
		\draw[rlinemb] (d3) to [out=230,in=300] (d4);
		
	\end{tikzpicture} 
	
	\caption{The graph $\cp{G}$ from Construction \ref{con:c:thet22k} such that $\ig{G} = \thet{2,2,5}$, with $\thet{2,2,5}$ overlaid in red.}
	\label{fig:c:thet225}
\end{figure}	
	

	As with our other constructions, the triangles of $\cp{G}$ are its smallest maximal cliques, and so $i(G)=3$. However, we now employ a technique of adding vertices to create $K_4$'s in $\cp{G}$ and eliminate any ``unwanted" triangles that might arise in our construction.  In \ref{con:c:thet22k:5}, the addition of $z_1$  and $z_2$  eliminate triangles $\{w_1,w_4,v_1\}$ and $\{w_2,w_3,v_2\}$, respectively.  Similarly, in   \ref{con:c:thet22k:ell}, $z$  prevents $\{w_1,w_4,v_{\ell-3}\}$ from being a maximal clique of $\cp{G}$ and hence an $i$-set of $G$.
	The unfortunate trade-off in this triangle-elimination technique is that the remaining triangles are no longer $\alpha$-sets; the constructions work only for $i$-graphs, not $\alpha$-graphs.

	\begin{lemma} \label{lem:c:thet22l}
		If $\cp{G}$ is the graph constructed by Construction \ref{con:c:thet22k}, then $\ig{G} = \thet{2,2,\ell}$, for $\ell \geq 5.$
	\end{lemma}

	\begin{proof}		
		We only prove the lemma for case where $\ell \geq 6$; the single case where $\ell = 5$ is adequately illustrated in Figure \ref{fig:c:thet225} and details can be found in \cite{LauraD}.

	As in the previous constructions, since each edge of $\cp{G}$ belongs to a triangle and some triangles are not contained in $K_4$'s, these triangles of $\cp{G}$ form the smallest maximal cliques of $\cp{G}$.  We label these triangles as in Figure \ref{fig:c:thet22k}; in particular,	
			
			\begin{center}	
				\begin{tabular}{rlcrl}
					$X$ & $=\{w_0,w_1,w_2\},$  & \null \hspace{20mm} \null &  	$D_1$ & $= \{w_1,w_2,v_1\},$ \\
					$Y$ & $= \{w_0,w_3,w_4\},$ &&   	$D_i$ & $= \{w_1,v_{i-1},v_i\}$  for $  2 \leq i \leq \ell-4,$\\
					$A$ & $=\{w_0,w_2,w_3\},$  &&		$D_{\ell-3}$ & $= \{v_{\ell-5}, v_{\ell-4}, v_{\ell-3} \},$ \\
					$B$ & $=\{w_0,w_1,w_4\},$  &&   	$D_{\ell-2}$ & $= \{w_{4}, v_{\ell-4}, v_{\ell-3} \},$ \\
					& && 							$D_{\ell-1}$ & $= \{w_3,w_4,v_{\ell-3}\}.$
				\end{tabular}
			\end{center}

   It is straightforward to verify that these $\ell+3$ sets are precisely the maximal cliques of $\cp{G}$ of order $3$ and, hence, the $i$-sets of $G$. Therefore they form the vertex set of $\ig{G}$. Moving to the edges of $\ig{G}$, the following adjacencies are clear from Figure \ref{fig:c:thet22k}:
			\begin{enumerate}[itemsep=0pt, label=(\roman*)]
				\item \label{lem:c:thet22l:i}  $\edge{X, w_1, w_3,A} \adedge{w_2,w_4,Y}$, 
				
				\item \label{lem:c:thet22l:j} $\edge{X, w_2, w_4,B} \adedge{w_1,w_3,Y}$, 
				
				\item \label{lem:c:thet22l:k} $\edge{X, w_0,  v_1,D_1}  \adedge{w_2,v_2,D_2} \adedge{v_1,v_3,D_3} \dots \edge{D_{\ell-3}, w_1,  v_{\ell-2},D_{\ell-2}} \adedge{v_{\ell-4},w_4,D_{\ell-1}}  \adedge{v_{\ell-2},w_0,Y}$. 
			\end{enumerate}

           As for its vertex set,  it is straightforward to verify that the edge set of $\ig{G}$ consists of precisely the edges listed in $\mathrm{(i)-(iii)}$.  We conclude that $\ig{G} = \thet{2,2,\ell}$.

	\end{proof}


	
	\subsubsection{Construction of $\cp{G}$ for $\thet{2,3,\ell}$ for $\ell \geq 5$}  \label{subsubsec:c:23k}

	For many of the results going forward, we apply small modifications to previous constructions.  In the first of these, we begin with the graphs from Construction \ref{con:c:thet22k}, which were used to find $i$-graphs for $\thet{2,2,5}$ and $\thet{2,2,\ell}$ for $\ell \geq 6$,  and expand the central wheel used in there  to build $i$-graphs for $\thet{2,3,5}$ and $\thet{2,3,\ell}$ for $\ell \geq 6$.

	\begin{cons} \label{con:c:thet23k}  Refer to Figures \ref{fig:c:thet23k} and \ref{fig:c:thet235}.	
		
		\begin{enumerate}[itemsep=0pt, label=(\alph*)]
			\item \label{con:c:thet23k:ell}
			If $\ell \geq 6$, begin with a copy of the graph $\cp{G}$ from Construction \ref{con:c:thet22k}.  Subdivide the edge $w_1w_4$, adding the new vertex $w_5$.  Join $w_5$ to $w_0$, so that $w_0,w_1,\dots,w_5$ forms a wheel.  Delete the vertex $z$.

			\item \label{con:c:thet23k:5} If $\ell=5$,	begin with a copy of the graph $\cp{G}$ from Construction \ref{con:c:thet22k}.  Subdivide the edge $w_1w_4$, adding the new vertex $w_5$.  Join $w_5$ to $w_0$, so that $w_0,w_1,\dots,w_5$ forms a wheel.  Delete the vertex $z_1$.
		\end{enumerate}
		
		We rename the resultant (planar) graph $\cp{G_{2,3,\ell}}$ for $\ell \geq 5$.	
	\end{cons}


\begin{figure}[H] \centering	
	\begin{tikzpicture}	[scale=0.8]
		\node (cent) at (0,0) {};
		\path (cent) ++(0:40 mm) node (rcent) {};	
		
		\node[std] (w0) at (cent) {};
			\path (cent) ++(25:5 mm) node (Lw0) {$w_{0}$};
		
		\foreach \i in {1,2,4}
		{
			\node[std,label={180-\i*90: $w_{\i}$}] (w\i) at (180-\i*90:2cm) {}; 
			\draw[thick] (w0)--(w\i);        
		}
	
		\node[std,label={0: $w_{3}$}] (w3) at (180-270:2cm) {}; 
		\draw[thick] (w0)--(w3);        
	
		\draw[thick] (w1)--(w2)--(w3)--(w4);

		
		\node[std,label={135: $w_5$}] (w5) at (135:2cm) {}; 
		\draw[thick] (w0)--(w5); 
		\draw[thick] (w4)--(w5)--(w1);

		\foreach \i / \j  / \k in {1/3/1, 2/2/2, 3/-1/{\ell-5}, 4/-2.2/{\ell-4}, 5/-3.4/{\ell-3}}
		{	
			\path (rcent) ++(90:\j cm) node [bblue] (v\i) {};
		}
		
		\path (v1) ++(20:4mm) node (Lv1) {\color{blue} $v_{1}$};
		\path (v2) ++(200:4 mm) node (Lv2) {\color{blue} $v_{2}$};
		\path (v3) ++(160:8 mm) node (Lv3) {\color{blue} $v_{\ell-5}$};
		\path (v4) ++(0:8 mm) node (Lv4) {\color{blue} $v_{\ell-4}$};
		\path (v5) ++(260:4 mm) node (Lv5) {\color{blue} $v_{\ell-3}$};		
		
		\draw[bline] (v1)--(v2);
		\draw[bline] (v3)--(v4)--(v5);
		
		\draw[bline,dashed] (v2)--(v3);
		
		\path (v4) ++(180:1.5 cm) coordinate  (c1) {};
		\draw[bline] (v3) to[out=190, in=90]  (c1) to [out=270,in=170] (v5);  
		
		\draw[bline] (w1)--(v1)--(w2);	
		
		\draw[bline] (v5) to[out=195,in=-90] (w4);  
		
		\draw[bline] (v5) to[out=180,in=-30] (w3);  
		
		\path (v5) ++(270: 10mm) coordinate  (cv4) {};		
		\draw[bline] (v4) to[out=-35, in=10]  (cv4) to [out=190,in=260] (w4);  	
		
		\draw[bline] (w1) to[out=25, in=180]  (40:5.8cm) to [out=350,in=45]  (v2); 
		
		\draw[bline] (w1) to[out=38, in=180]  (40:6.4cm) to [out=350,in=45]  (v3); 
		
		\draw[bline] (w1) to[out=40, in=180]  (40:6.6cm) to [out=350,in=40]  (v4); 		
		
		\path (rcent) ++(90:1 cm) coordinate  (t1) {};
		\draw[bline, dashed] (w1) to[out=26, in=180]   (40:6cm)  to [out=350,in=45]  (t1); 
		
		\path (rcent) ++(90:0 cm) coordinate  (t2) {};
		\draw[bline, dashed] (w1) to[out=31, in=180]  (40:6.2cm) to [out=350,in=45]  (t2); 
		
		\path (w0) ++(45:9 mm) node [color=red] (X) {$X$};			
		\path (w0) ++(45:-9 mm) node [color=red] (Y) {$Y$};
		
		\path (w0) ++(-45:8 mm) node [color=red] (A1) {$A$};	
		
		\path (w0) ++(113:12 mm) node [color=red] (B1) {$B_1$};	
		\path (w0) ++(158:12 mm) node [color=red] (B2) {$B_2$};	
		
		\path (w2) ++(90:15 mm) node [color=red] (D1) {$D_1$};	
		\path (v1) ++(170:9 mm) node [color=red] (D2) {$D_2$};	
		\path (v4) ++(180:8 mm) node [color=red] (D3) {$D_{\ell-3}$};
		\path (v5) ++(195:20mm) node [color=red] (D4) {$D_{\ell-2}$};
		\path (w3) ++(270:5 mm) node [color=red] (D5) {$D_{\ell-1}$};		
		
	\end{tikzpicture} 
	\caption{The graph $\cp{G_{2,3,\ell}}$ from Construction \ref{con:c:thet23k} \ref{con:c:thet23k:ell} such that $\ig{G_{2,3,\ell}} = \ag{G_{2,3,\ell}} =  \thet{2,3,\ell}$ for $\ell \geq 6$.}
	\label{fig:c:thet23k}
\end{figure}	

\begin{figure}[H] \centering	
	\begin{tikzpicture}	[scale=0.8]			
		\node (cent) at (0,0) {};
		\path (cent) ++(0:45 mm) node (rcent) {};	
		
		\node[std] (w0) at (cent) {};
		\path (w0) ++(30: 4 mm) node (Lw0) {${w_0}$};	
		
		\foreach \i / \rad in {1/90,2/0,3/300,4/180}
		{
			\node[std,label={\rad: $w_{\i}$}] (w\i) at (180-\i*90:2.5cm) {}; 
			\draw[thick] (w0)--(w\i);        
		}
		\draw[thick] (w1)--(w2)--(w3)--(w4);
		
		\node[std] (w5) at (135:2.2cm) {}; 
			\path (w5) ++(73: 5 mm) node (Lw5) {${w_5}$};
		\draw[thick] (w0)--(w5); 
		\draw[thick] (w4)--(w5)--(w1);				
		
		\foreach \i / \j / \rad in {1/1.5/0,2/-1.5/270}
		{	
			\path (rcent) ++(90:\j cm) node [std] (v\i) {};
			\path (v\i) ++(\rad:5 mm) node (Lv\i) {$v_{\i}$};
		}
		\draw[thick] (w1)--(v1)--(v2)--(w3);
		\draw[thick] (v2)--(w2)--(v1);	
		
		\path (w2) ++(270:1 cm) node [bzed] (z2) {};
			\path (z2) ++(270:4 mm) node (Lz2) {\color{zelim}$\boldsymbol{z_{2}}$};
		\draw[zline] (w2)--(z2)--(w3);
		\draw[zline] (z2)--(v2);			
		
		\draw[thick] (v2) to[out=230, in=-5]  (270:3.3cm) to [out=170,in=280] (w4);
		
		\draw[thick] (v1) to[out=130, in=5]  (90:3.3cm) to [out=190,in=80] (w4);	
		
		\foreach \i / \j / \rad in {1/X/90,3/Y/160,4/A/0}
		{	
			\node[dred] (t\i) at (-45+\i*90:1.0cm) {}; 
			\path (t\i) ++(\rad: 4 mm) node (Lt\i) {\color{red}${\j}$};
			
		}
		
		\path (w0) ++(113:10 mm) node [dred] (t2a) {};	
			\path (t2a) ++(113: 6 mm) node (Lt2a) {\color{red}${B_1}$};
		\path (w0) ++(158:10 mm) node [dred] (t2b) {};	
			\path (t2b) ++(158: 6 mm) node (Lt2b) {\color{red}${B_2}$};		
		
		\path (w2) ++(90: 7mm) node [dred] (d1) {};
		\path (d1) ++(90: 4 mm) node (Ld1) {\color{red}${D_1}$};
		
		\path (w2) ++(0: 15mm) node [dred] (d2) {};
		\path (d2) ++(90: 4 mm) node (Ld2) {\color{red}${D_2}$};	
		
		\path (v2) ++(-45: 10mm) node [dred] (d3) {};
		\path (d3) ++(0: 5 mm) node (Ld3) {\color{red}${D_3}$};	
		
		\path (t3) ++(260: 15mm) node [dred] (d4) {};
		\path (d4) ++(140: 6mm) node (Ld4) {\color{red}${D_4}$};					
		
		\draw[rlinemb](t1)--(t2a)--(t2b)--(t3)--(t4)--(t1);
		\draw[rlinemb](t1)--(d1)--(d2);	
		\draw[rlinemb](t3)--(d4);
		
		\draw[rlinemb] (d2) to [out=-45,in=90] (d3);
		
		\draw[rlinemb] (d3) to [out=230,in=300] (d4);
		
	\end{tikzpicture} 
	
	\caption{The graph $\cp{G_{2,3,5}}$ from Construction \ref{con:c:thet23k} \ref{con:c:thet23k:5} such that $\ig{G_{2,3,5}} = \thet{2,3,5}$, with $\ig{G_{2,3,5}}$ overlaid in red.}
	\label{fig:c:thet235}
\end{figure}	

	In Construction \ref{con:c:thet23k} \ref{con:c:thet23k:ell}, notice that the vertex $z$ is deleted from $\cp{G}$.    In the original Construction \ref{con:c:thet22k} for a graph $\cp{G}$ with $\ig{G} = \thet{2,2,\ell}$ for $\ell \geq 6$, $z$ served to eliminate the unwanted triangle formed by $\{w_1,w_4, v_{\ell-4}\}$.   Now with the expanded wheel including $w_5$, $\{w_1,w_4, v_{\ell-4}\}$ is not a triangle in $\cp{G_{2,3,\ell}}$ ($\ell \geq 6$), and $z$ is not needed.  Indeed, as $\cp{G_{2,3,\ell}}$ now has $\alpha({G_{2,3,\ell}}) = 3$, its triangles are also $\alpha$-sets in $G$, and so we can immediately extend the construction from $i$-graphs to $\alpha$-graphs.
	
	The extension, however, does not apply to  Construction \ref{con:c:thet23k} \ref{con:c:thet23k:5} for the graph $\cp{G_{2,3,5}}$.  Here, we no longer require  $z_1$ (which served to eliminate the unwanted triangle formed by $\{w_1,w_4, v_1\}$), but $z_2$ remains and forms the clique $\{w_2,w_3,z_2,v_2\}$; thus, $\alpha({G_{2,3,5}}) = 4$.  
	
	
	The proof for the  following Lemma \ref{lem:c:thet23k} is otherwise very similar to the proof of Lemma \ref{lem:c:thet22l}, and so is omitted.

	
	\begin{lemma}\label{lem:c:thet23k}
		If $\cp{G_{2,3,5}}$ is the graph constructed in Construction \ref{con:c:thet23k} \ref{con:c:thet23k:5}, then $\ig{G_{2,3,5}} = \thet{2,3,5}$.  
		For $\ell \geq 6$, if $\cp{G_{2,3,\ell}}$  is the graph constructed in Construction \ref{con:c:thet23k} \ref{con:c:thet23k:ell}, then $\ig{G_{2,3,\ell}} = \ag{G_{2,3,\ell}} = \thet{2,3,\ell}$ for $\ell \geq 6$.  	
	\end{lemma}
	


	
	\subsubsection{Construction of $\cp{G}$ for $\thet{2,4,4}$}  \label{subsubsec:c:244}

	In the following construction for a graph $\cp{G}$ with $\ig{G} = \thet{2,4,4}$, we again apply the technique of  adding vertices to eliminate unwanted triangles.
	
	\begin{figure}[H] \centering	
		\begin{tikzpicture}	[scale=0.8]	
			\node (cent) at (0,0) {};
			\path (cent) ++(0:50mm) node (rcent) {};	
			
			\node[std,label={0:$w_0$}] (w0) at (cent) {};
			
			\foreach \i in {1,2,3,4,5,6}
			{
				\node[std,label={150-\i*60: $w_{\i}$}] (w\i) at (150-\i*60:2cm) {}; 
				\draw[thick] (w0)--(w\i);        
			}
			\draw[thick] (w1)--(w2)--(w3)--(w4)--(w5)--(w6)--(w1);
			
			\path (rcent) ++(90:0 cm) node [bblue] (v) {};
			\path (v) ++(0:4 mm) node (Lv) {\color{blue} $v$};	
			
			\foreach \i / \j in {2,3}
			\draw[bline] (v)--(w\i);	
			
			\draw[bline] (w1)   to[out=0,in=100] (v); 	
			\draw[bline] (w4)   to[out=0,in=-100] (v); 
			\draw[bline] (w4) to[out=-15, in=270]  (0:6.5cm)  to[out=90,in=15] (w1); 		
			
			\path (w0) ++(0:3.5 cm) node [bzed, label={[zelim]180:$\boldsymbol{z}$}] (z) {};
			\draw[zline] (w2)--(z)--(w3);
			\draw[zline] (z)--(v);
			
			\path (w0) ++(180:4 cm) node [bzed, label={[zelim]180:$\boldsymbol{z'}$}] (z1) {};
			\draw[zline] (w0)   to[out=165,in=15] (z1); 
			\draw[zline] (w1)   to[out=180,in=80] (z1); 		
			\draw[zline] (w4)   to[out=180,in=-80] (z1); 	
			
			\path (w0) ++(60:12 mm) node [color=red] (X) {$X$};	
			\path (w0) ++(0:12 mm) node [color=red] (A) {$A$};	
			\path (w0) ++(-60:12 mm) node [color=red] (Y) {$Y$};	
			
			\foreach \i/\c in {1/120,2/180,3/240} {				
				\path (w0) ++(\c:12 mm) node [color=red] (B\i) {$B_{\i}$};	
			}
			
			\path (w2) ++(05:15 mm) node [color=red] (D1) {$D_1$};	
			\path (v) ++(60:8 mm) node [color=red] (D2) {$D_2$};	
			\path (w3) ++(-05:15 mm) node [color=red] (D3) {$D_3$};
			
		\end{tikzpicture} 
		\caption{A graph $\cp{G}$ such that $\ig{G} = \thet{2,4,4}$.}
		\label{fig:c:thet244}
	\end{figure}	

	\begin{cons} \label{con:c:thet244}
		Refer to Figure \ref{fig:c:thet244}.	Begin with a copy of the graph $\cp{H} \cong W_7 = C_6 \vee K_1$, labelling the degree $3$ vertices as $w_1,w_2,\dots, w_6$ and the central degree $6$ vertex as $w_0$.  Join $w_1$ to $w_4$.    Add a new vertex $v$ to $\cp{H}$, joining $v$ to $w_1, \dots, w_4$.    Then, add the new vertex $z$, joined to $v$, $w_2$ and $w_3$, and the new vertex $z'$, joined to $w_0, w_1$, and $w_4$.  We label the resultant (non-planar) graph $\cp{G}$.	
	\end{cons}

	\begin{lemma}\label{lem:c:thet244}
		If $\cp{G}$ is the graph constructed by Construction \ref{con:c:thet244}, then $\ig{G} = \thet{2,4,4}$.
	\end{lemma}

	\begin{proof}	
		As shown in Figure \ref{fig:c:thet244}, the triangles of $\cp{G}$ forming maximal cliques of $\cp{G}$, and therefore the $i$-sets of $G$, are labelled them as follows:	
		
		\begin{center}	
			\begin{tabular}{rlcrlcrl}
				$X$ & $=\{w_0,w_1,w_2\},$  & \null \hspace{1mm} \null &  	$B_1$ & $= \{w_0,w_1,w_6\},$ & \null \hspace{1mm} \null &  	$D_1$ & $= \{w_1,w_2,v\},$ \\
				$Y$ & $= \{w_0,w_3,w_4\},$ &&   	$B_2$ & $= \{w_0,w_5,w_6\}$,  & \null \hspace{1mm} \null &  	$D_2$ & $= \{w_1,w_4,v\},$ \\
				$A$ & $=\{w_0,w_2,w_3\},$  &&		$B_3$ & $= \{w_0,w_4,w_5 \},$ & \null \hspace{1mm} \null &  	$D_3$ & $= \{w_3,w_4,v\}.$ \		
			\end{tabular}
		\end{center}
		
		Similarly to the construction for $\thet{2,2,5}$  in the proof of Lemma \ref{lem:c:thet22l}, the vertices $z$ and $z'$ are added to ensure that $\{v,w_2,w_3\}$ and $\{w_0,w_1,w_4\}$, respectively, are not maximal cliques in  $\cp{G}$, and hence are not $i$-sets of $G$.
		
		It can be seen that there are no other triangles in $\cp{G}$ beyond the nine listed above; we omit the details, which can be found in \cite[Lemma 5.19]{LauraD}.

		From Figure \ref{fig:c:thet244}, the following triangle adjacencies are immediate:
		\begin{enumerate}[itemsep=0pt, label=(\roman*)]
			\item \label{lem:c:thet244:i}  $\edge{X, w_1, w_3,A} \adedge{w_2,w_4,Y}$, 
			
			\item \label{lem:c:thet244:j} $\edge{X, w_2, w_6,B_1} \adedge{w_1,w_5,B_2} \adedge{w_6,w_4,B_3} \adedge{w_5,w_3,Y} $, 
			
			\item \label{lem:c:thet244:k} $\edge{X, w_0,  v,D_1}  \adedge{w_2,w_4,D_2} \adedge{w_1,w_3,D_3}  \adedge{v,w_0,Y}$. 
		\end{enumerate}

		
  
       It is again straightforward to verify that there are no additional edges in $\ig{G}$ than those listed above. We conclude that $\ig{G} = \thet{2,4,4}$ as required.
	\end{proof}
	
	\medskip

	Notice that Construction \ref{con:c:thet244}  is our first theta graph construction that is not planar as it has a $K_{3,3}$ minor.  Indeed, with the exception of the constructions that are based upon Construction \ref{con:c:thet244}, all of our $i$-graph constructions use planar complement seed graphs.

	\begin{problem}	\null 	
		\begin{enumerate}[label=(\roman*)]
			\item Find a planar graph-complement construction for $\thet{2,4,4}$.  
			
			\item 	 Do all $i$-graphs with largest induced stars of $K_{1,3}$, always have a planar graph comple-ment construction?
			
			A large target graph requires a large seed graph in order to generate a sufficient number of unique $i$-sets.   Can a target graph become too dense to allow for a planar graph-complement construction?	
		\end{enumerate}
		
	\end{problem}

	Moving forward, we will no longer explicitly check that there are no additional unaccounted for triangles in our constructions.  Should the construction indeed result in triangles of $\cp{G}$ that produce extraneous vertices in $\ig{G}$, we can easily remove them using the Deletion Lemma (Lemma \ref{lem:i:inducedI}).   
 
	

	
	\subsubsection{Constructions of $\cp{G}$ for $\thet{2,4,5}$ and $\thet{2,5,5}$}  \label{subsubsec:c:245}
	
	

	\begin{cons} \label{cons:c:thet2k5}	  Refer to Figure  \ref{fig:c:thet255}.
		\begin{enumerate}[itemsep=0pt, label=(\alph*)]
			\item \label{con:c:thet245:a}
			Begin with a copy of the graph $\cp{G_{2,3,5}}$ from Construction \ref{con:c:thet23k} \ref{con:c:thet23k:5} for $\thet{2,3,5}$.  Subdivide the edge $w_1w_5$, adding the new vertex $w_6$.  Join $w_6$ to $w_0$, so that $w_0,w_1,\dots,w_6$ forms a wheel.  Call this graph $\cp{G_{2,4,5}}$.

			\item \label{con:c:thet255:b} Begin with a copy of the graph $\cp{G_{2,4,5}}$  from  Construction \ref{cons:c:thet2k5} \ref{con:c:thet245:a}.  Subdivide the edge $w_1w_6$, adding the new vertex $w_7$.  Join $w_7$ to $w_0$, so that $w_0,w_1,\dots,w_7$ forms a wheel.  Call this graph $\cp{G_{2,5,5}}$.  
		\end{enumerate}
	\end{cons}

\begin{figure}[H] \centering	
	\begin{tikzpicture}	[scale=0.8]
		
		\node (cent) at (0,0) {};
		\path (cent) ++(0:45 mm) node (rcent) {};	
		
		\node[std] (w0) at (cent) {};
		\path (w0) ++(30: 4 mm) node (Lw0) {${w_0}$};	
		
		\foreach \i in {1,2,4}
		{
			\node[std,label={180-\i*90: $w_{\i}$}] (w\i) at (180-\i*90:2.5cm) {}; 
			\draw[thick] (w0)--(w\i);        
		}
	
		\node[std,label={310: $w_{3}$}] (w3) at (180-270:2.5cm) {}; 
		\draw[thick] (w0)--(w3);

		\draw[thick] (w1)--(w2)--(w3)--(w4);
		
		\node[std] (w5) at (157.5:2.5cm) {}; 
		\path (w5) ++(94: 5 mm) node (Lw5) {${w_5}$};	
		\draw[thick] (w0)--(w5); 
		
		\node[std,label={90: $w_6$}] (w6) at (135:2.5cm) {}; 
		\draw[thick] (w0)--(w6); 
		
		\node[std,label={90: $w_7$}] (w7) at (112.5:2.5cm) {}; 
		\draw[thick] (w0)--(w7); 
		\draw[thick] (w4)--(w5)--(w6)--(w7)--(w1);

		\foreach \i / \j / \rad in {1/1.5/0,2/-1.5/270}
		{	
			\path (rcent) ++(90:\j cm) node [std] (v\i) {};
			\path (v\i) ++(\rad:5 mm) node (Lv\i) {$v_{\i}$};
		}
		\draw[thick] (w1)--(v1)--(v2)--(w3);
		\draw[thick] (v2)--(w2)--(v1);	
		
		\path (w2) ++(270:1 cm) node [bzed] (z2) {};
			\path (z2) ++(270:4 mm) node (Lz2) {\color{zelim}$\boldsymbol{z_{2}}$};
		\draw[zline] (w2)--(z2)--(w3);
		\draw[zline] (z2)--(v2);	
		
		\draw[thick] (v2) to[out=230, in=-5]  (270:3.3cm) to [out=170,in=280] (w4);
		
		\draw[thick] (v1) to[out=130, in=5]  (110:3.5cm) to [out=195,in=110] (w4);	
		


		\foreach \i / \j / \rad in {1/X/90,3/Y/160,4/A/0}
		{	
			\node[dred] (t\i) at (-45+\i*90:1.0cm) {}; 
			\path (t\i) ++(\rad: 5 mm) node (Lt\i) {\color{red}${\j}$};
			
		}	
		
		\path (w0) ++(101.25:10 mm) node [dred] (t2a) {};	
		\path (t2a) ++(101.25: 6 mm) node (Lt2a) {\color{red}${B_1}$};		
		\path (w0) ++(123.75:10 mm) node [dred] (t2b) {};
		\path (t2b) ++(123.75: 6 mm) node (Lt2b) {\color{red}${B_2}$};	
		\path (w0) ++(146.25:10 mm) node [dred] (t2c) {};	
		\path (t2c) ++(146.25: 6 mm) node (Lt2c) {\color{red}${B_3}$};	
		\path (w0) ++(168.75:10 mm) node [dred] (t2d) {};	
		\path (t2d) ++(168.75: 6 mm) node (Lt2d) {\color{red}${B_4}$};

		\path (w2) ++(90: 7mm) node [dred] (d1) {};
		\path (d1) ++(90: 4 mm) node (Ld1) {\color{red}${D_1}$};
		
		\path (w2) ++(0: 15mm) node [dred] (d2) {};
		\path (d2) ++(90: 4 mm) node (Ld2) {\color{red}${D_2}$};	
		
		\path (v2) ++(-45: 10mm) node [dred] (d3) {};
		\path (d3) ++(0:5mm) node (Ld3) {\color{red}${D_3}$};	
		
		\path (t3) ++(260: 15mm) node [dred] (d4) {};
		\path (d4) ++(145: 6 mm) node (Ld4) {\color{red}${D_4}$};

		\draw[rlinemb](t1)--(t2a)--(t2b)--(t2c)--(t2d)--(t3)--(t4)--(t1);
		\draw[rlinemb](t1)--(d1)--(d2);	
		\draw[rlinemb](t3)--(d4);
		
		\draw[rlinemb] (d2) to [out=-45,in=90] (d3);
		
		\draw[rlinemb] (d3) to [out=230,in=300] (d4);
		
	\end{tikzpicture} 
	
	\caption{The graph $\cp{G_{2,5,5}}$ from Construction \ref{cons:c:thet2k5}	 such that $\ig{G} = \thet{2,5,5}$, with $\ig{G}$ overlaid in red.}
	\label{fig:c:thet255}
\end{figure}	

	\begin{lemma} \label{lem:c:thet2k5}	
		If $\cp{G_{2,k,5}}$ is the graph constructed by Construction \ref{cons:c:thet2k5}, then $\ig{G_{2,k,5}} = \thet{2,k,5}$, for $4 \leq k \leq 5$.
	\end{lemma}

	\noindent Lemma \ref{lem:c:thet2k5} follows readily from Figure \ref{fig:c:thet255} and we omit the proof.
	
	
	\subsubsection{Construction of $\cp{G}$ for $\thet{2,k,\ell}$ for $\ell \geq k \geq 4$ and $\ell \geq 6$}  \label{subsubsec:c:2kl}

	
	\begin{cons}\label{cons:c:thet2kl} 
		See Figure \ref{fig:c:thet2kl}.   Begin with a copy of the graph $\cp{G_{2,3,\ell}}$ from Construction \ref{con:c:thet23k} \ref{con:c:thet23k:ell} for $\ell\geq 5$.  Subdivide the edge $w_1w_5$ $k-3$ times (for $k \leq \ell$), adding the new vertices $w_6$, $w_7, \dots, w_{k+2}$.  Join $w_6, w_7, \dots, w_{k+2}$ to $w_0$, so that $w_0,w_1,\dots,w_{k+2}$ forms a wheel.  Call this graph $\cp{G_{2,k,\ell}}$.		
	\end{cons}

\begin{figure}[H] \centering	
	\begin{tikzpicture}	[scale=0.8]
		\node (cent) at (0,0) {};
		\path (cent) ++(0:40 mm) node (rcent) {};	
		
		\node[std] (w0) at (cent) {};
		\path (w0) ++(25: 5 mm) node (Lw0) {${w_0}$};	
		
		\foreach \i in {1,2,4}
		{
			\node[std,label={180-\i*90: $w_{\i}$}] (w\i) at (180-\i*90:2cm) {}; 
			\draw[thick] (w0)--(w\i);        
		}
	
		\node[std,label={0: $w_{3}$}] (w3) at (180-270:2cm) {}; 
		\draw[thick] (w0)--(w3);     
	
		\draw[thick] (w1)--(w2)--(w3)--(w4);		
		
		\node[std] (w5) at (157.5:2cm) {}; 
		\path (w5) ++(157: 5 mm) node (Lw5) {${w_5}$};	
		\draw[thick] (w0)--(w5); 		
		
		\node[std,label={112: $w_{k-3}$}] (w7) at (112.5:2cm) {}; 
		\draw[thick] (w0)--(w7); 
		\draw[thick] (w4)--(w5);
		\draw[thick] (w7)--(w1);	
		\draw[thick,dashed](w5)--(w7);		
		
		\coordinate  (c1) at ($(w5)!0.33!(w7)$) {};
		\coordinate  (c2) at ($(w5)!0.66!(w7)$) {};
		
		\draw [dashed, thick] (c2)--(w0);
		\draw [dashed, thick] (c1)--(w0);				
		
		\foreach \i / \j  / \k in {1/3/1, 2/2/2, 3/-1/{\ell-5}, 4/-2.2/{\ell-4}, 5/-3.4/{\ell-3}}
		{	
			\path (rcent) ++(90:\j cm) node [bblue] (v\i) {};
		}
		
		\path (v1) ++(0:4 mm) node (Lv1) {\color{blue} $v_{1}$};
		\path (v2) ++(200:4 mm) node (Lv2) {\color{blue} $v_{2}$};
		\path (v3) ++(150:6 mm) node (Lv3) {\color{blue} $v_{\ell-5}$};
		\path (v4) ++(0:8 mm) node (Lv4) {\color{blue} $v_{\ell-4}$};
		\path (v5) ++(260:4 mm) node (Lv5) {\color{blue} $v_{\ell-3}$};		
		
		\draw[bline] (v1)--(v2);
		\draw[bline] (v3)--(v4)--(v5);
		
		\draw[bline,dashed] (v2)--(v3);		
		
		\path (v4) ++(180:1.5 cm) coordinate  (c1) {};
		\draw[bline] (v3) to[out=190, in=90]  (c1) to [out=270,in=170] (v5);  
		
		\draw[bline] (w1)--(v1)--(w2);	
		
		\draw[bline] (v5) to[out=195,in=-90] (w4);  
		
		\draw[bline] (v5) to[out=180,in=-30] (w3);  
		
		\path (v5) ++(270: 10mm) coordinate  (cv4) {};		
		\draw[bline] (v4) to[out=-35, in=10]  (cv4) to [out=190,in=260] (w4);  	
		
		\draw[bline] (w1) to[out=25, in=180]  (40:5.8cm) to [out=350,in=45]  (v2); 
		
		\draw[bline] (w1) to[out=38, in=180]  (40:6.4cm) to [out=350,in=45]  (v3); 
		
		\draw[bline] (w1) to[out=40, in=180]  (40:6.6cm) to [out=350,in=40]  (v4); 		
		
		\path (rcent) ++(90:1 cm) coordinate  (t1) {};
		\draw[bline, dashed] (w1) to[out=26, in=180]   (40:6cm)  to [out=350,in=45]  (t1); 
				
		\path (rcent) ++(90:0 cm) coordinate  (t2) {};
		\draw[bline, dashed] (w1) to[out=31, in=180]  (40:6.2cm) to [out=350,in=45]  (t2); 
		
		\path (w0) ++(45:9 mm) node [color=red] (X) {$X$};			
		\path (w0) ++(45:-9 mm) node [color=red] (Y) {$Y$};
		
		\path (w0) ++(-45:8 mm) node [color=red] (A1) {$A$};			
		
		\path (w0) ++(100:15 mm) node [color=red] (B1) {\scalebox{0.9}{$B_1$}};	
		\path (w0) ++(170:14 mm) node [color=red] (B2) {$B_{k-1}$};		
		
		\path (w2) ++(90:15 mm) node [color=red] (D1) {$D_1$};	
		\path (v1) ++(170:9 mm) node [color=red] (D2) {$D_2$};	
		\path (v4) ++(180:8 mm) node [color=red] (D3) {$D_{\ell-3}$};
		\path (v5) ++(195:20mm) node [color=red] (D4) {$D_{\ell-2}$};
		\path (w3) ++(270:5 mm) node [color=red] (D5) {$D_{\ell-1}$};			
		
	\end{tikzpicture} 
	\caption{The graph $\cp{G_{2,k,\ell}}$ from Construction \ref{cons:c:thet2kl} such that $\ig{\cp{G_{2,k,\ell}}} = \thet{2,k,\ell}$ for $\ell \geq k \geq 4$ and $\ell \geq 6$.}
	\label {fig:c:thet2kl}
\end{figure}	

	\noindent Lemma \ref{lem:c:thet2kl} follows from Figure \ref{fig:c:thet2kl}.
	
	\begin{lemma} \label{lem:c:thet2kl}
		If $\cp{G_{2,k,\ell}}$ is the graph constructed by Construction \ref{cons:c:thet2kl}, then $\ig{G_{2,k,\ell}} = \ag{G_{2,k,\ell}} = \thet{2,k,\ell}$, for $\ell \geq k \geq 4$ and $\ell \geq 6$.
	\end{lemma}

	\subsection{$\thet{3,k,\ell}$}  \label{subsec:c:3kl}
	
	In the rest of Section \ref{sec:c:thetaComp} we only give the constructions and their associated figures, and state the lemmas without proof. Details can be found in \cite[Chapter 5]{LauraD}.
	
	
	\subsubsection{The graph $\cp{G}$ for $\thet{3,3,4}$}  \label{subsubsec:c:334}

 The triangles $X, A_1, A_2, Y, B_2, B_1, D_1, D_2, D_3$ in Figure \ref{fig:c:thet334comp}, and the obvious paths formed by them, illustrate the following lemma.

	\begin{lemma} \label{lem:c:thet334}
		$\thet{3,3,4}$ is an $i$-graph.
	\end{lemma}

\begin{figure}[H] \centering	
	\begin{tikzpicture}	[scale=0.8]
		\node (cent) at (0,0) {};
		\path (cent) ++(0:40 mm) node (rcent) {};	
		
		\node[std,label={0:$w_0$}] (w0) at (cent) {};
		
		\foreach \i / \j in {1/2,2/6,3/1,4/3,5/5,6/4}
		{
			\node[std] (w\i) at (150-\i*60:2cm) {}; 
				\path (w\i) ++(150-\i*60:5 mm) node (Lw\i) {$v_{\j}$};
			\draw[thick] (w0)--(w\i);        
		}
		\draw[thick] (w1)--(w2)--(w3)--(w4)--(w5)--(w6)--(w1);
		
		\path (rcent) ++(90:-0.5 cm) node [bblue,label={[blue]0: $v_7$}] (v) {};	
		
		\draw[bline] (v)--(w2);
		
		\draw[bline] (v) to[out=90,in=0] (w1);

		\draw[bline] (v) to[out=230,in=0] (w4);  	
		
		\draw[bline] (v) to [out=260,in=15]  (-60:3.5cm) to [out=195,in=270] (w5);  
		
		\draw[bline] (w2) to [out=90, in=20]  (120:3cm) to [out=200,in=160] (w5); 
		
		\path (w0) ++(120:5 mm) node [bzed] (z) {};
				\path (z) ++(90:3 mm) node (Lz) {\color{zelim}$\boldsymbol{v_0}$};
		\draw[zline] (z)--(w0);
		\draw[zline] (w5)--(z)--(w2);

		\path (w0) ++(60:12 mm) node [color=red] (X) {$X$};			
		\path (w0) ++(60:-12 mm) node [color=red] (Y) {$Y$};
		
		\path (w0) ++(0:13 mm) node [color=red] (A1) {$A_1$};	
		\path (w0) ++(-60:12 mm) node [color=red] (A2) {$A_2$};
		
		\path (w0) ++(-60:-14 mm) node [color=red] (B1) {$B_1$};	
		\path (w0) ++(180:13 mm) node [color=red] (B2) {$B_2$};
		
		\path (w2) ++(-10:15mm) node [color=red] (D1) {$D_1$};	
		\path (v) ++(280:10 mm) node [color=red] (D2) {$D_2$};
		\path (w4) ++(-25:15 mm) node [color=red] (D3) {$D_3$};
		
	\end{tikzpicture} 
	\caption{A graph $\cp{G}$ such that $\ig{G} = \thet{3,3,4}$.}
	\label{fig:c:thet334comp}
\end{figure}
	\subsubsection{Construction of $\cp{G}$ for $\thet{3,3,5}$}  \label{subsubsec:c:335}

	\begin{cons} \label{con:c:thet335}
		Refer to  Figure \ref{fig:c:thet335}.	Begin with a copy of the graph $\cp{H} \cong W_7 = C_6 \vee K_1$, labelling the degree $3$ vertices as $w_1,w_2,\dots, w_6$ and the central degree $6$ vertex as $w_0$.     Add new vertices $v_1$ and $v_2$ to $\cp{H}$, joined to each other.  Then, join $v_1$ to each of $\{w_1,w_2,w_5\}$, and $v_2$ to each of $\{w_2, w_4,w_5\}$.  Call this graph $\cp{G_{3,3,5}}$.
	\end{cons}

\begin{figure}[H] \centering	
	\begin{tikzpicture}	[scale=0.8]
		\node (cent) at (0,0) {};
		\path (cent) ++(0:40 mm) node (rcent) {};	
		
		\node[std,label={0:$w_0$}] (w0) at (cent) {};
		
		\foreach \i in {1,4,5,6}
		{
			\node[std,label={150-\i*60: $w_{\i}$}] (w\i) at (150-\i*60:2cm) {}; 
			\draw[thick] (w0)--(w\i);        
		}
		
		\node[std] (w2) at (150-120:2cm) {}; 
				\path (w2) ++(-10:7 mm) node (Lw2) {$w_2$};		
		
		\node[std] (w3) at (-30:2cm) {}; 
			\path (w3) ++(-20:7 mm) node (Lw3) {$w_3$};			
			
		\draw[thick] (w0)--(w3);        		
		\draw[thick] (w0)--(w2);

		\draw[thick] (w1)--(w2)--(w3)--(w4)--(w5)--(w6)--(w1);
		
		\foreach \i/\j in {1/1.7, 2/-0.5}
		{
			\path (rcent) ++(90:\j cm) node [bblue,label={[blue]0: $v_\i$}] (v\i) {};
		}	
		
		\draw[bline] (w1)--(v1)--(v2)--(w2)--(v1);
		
		\draw[bline] (v2) to[out=230,in=0] (w4);  	
		
		\draw[bline] (v2) to [out=260,in=15]  (-60:3.5cm) to [out=195,in=270] (w5);  
		
		\draw[bline] (v1) to [out=-45, in=70]  (-15:5cm) to [out=250,in=15]  (-60:4cm) to[out=195,in=260] (w5); 
		
		\path (w0) ++(60:12 mm) node [color=red] (X) {$X$};			
		\path (w0) ++(60:-12 mm) node [color=red] (Y) {$Y$};
		
		\path (w0) ++(0:13 mm) node [color=red] (A1) {$A_1$};	
		\path (w0) ++(-60:12 mm) node [color=red] (A2) {$A_2$};
		
		\path (w0) ++(-60:-12 mm) node [color=red] (B1) {$B_1$};	
		\path (w0) ++(180:12 mm) node [color=red] (B2) {$B_2$};
		
		\path (w2) ++(100:5 mm) node [color=red] (D1) {$D_1$};	
		\path (w2) ++(-10:17mm) node [color=red] (D2) {$D_2$};	
		\path (v2) ++(280:10 mm) node [color=red] (D3) {$D_3$};
		\path (w4) ++(-25:15 mm) node [color=red] (D4) {$D_4$};	
		
	\end{tikzpicture} 
	\caption{A graph $\cp{G_{3,3,5}}$ from Construction \ref{con:c:thet335} such that $\ig{G_{3,3,5}} =$ $\thet{3,3,5}$.}
	\label{fig:c:thet335}
\end{figure}		

	\begin{lemma}\label{lem:c:thet335}
		If $\cp{G_{3,3,5}}$ is the graph constructed by Construction \ref{con:c:thet335}, then $\ig{G_{3,3,5}} = \ag{G_{3,3,5}} = \thet{3,3,5}$.
	\end{lemma}

	
	\subsubsection{Construction of $\cp{G}$ for $\thet{3,3,\ell}$ for $\ell \geq 6$}  \label{subsubsec:c:33k}

	\begin{cons} \label{con:c:thet33k}
		Refer to Figure \ref{fig:c:thet33k}.	Begin with a copy of the graph $\cp{H} \cong W_7 = C_6 \vee K_1$, labelling the degree $3$ vertices as $w_1,w_2,\dots, w_6$ and the central degree $4$ vertex as $w_0$.     For $\ell \geq 6$, add to $\cp{H}$  a new path of $\ell-3$ vertices labelled as $v_1$, $v_2, \dots, v_{\ell-3}$.  Then, join $w_1$ to each of $\{v_1,v_2,\dots, v_{\ell-4}\}$, $w_4$ to $v_{\ell-3}$, and $w_5$ to $v_{\ell-4}$ and $v_{\ell-3}$.  Finally, join $v_{\ell-5}$ to $v_{\ell-3}$, so that $\{v_{\ell-5}, v_{\ell-4}, v_{\ell-3}\}$ form a  $K_3$.  
		Call this graph $\cp{G_{3,3,\ell}}$ for $\ell \geq 6$.
	\end{cons}

	Although the general construction still applies for the case when $\ell=6$, we include a separate figure for the construction of $\thet{3,3,6}$ for reference below, because of the additional complication that now $v_1 = v_{\ell-5}$, and so the single vertex now has dual roles in the construction.

\begin{figure}[H] \centering	
	\begin{tikzpicture}	[scale=0.8]
		
		\node (cent) at (0,0) {};
		\path (cent) ++(0:40 mm) node (rcent) {};	
		
		\node[std,label={0:$w_0$}] (w0) at (cent) {};
		
		\foreach \i in {1, 5,6 }
		{
			\node[std,label={150-\i*60: $w_{\i}$}] (w\i) at (150-\i*60:2cm) {}; 
			\draw[thick] (w0)--(w\i);        
		}
	
		\foreach \i / \j in {2/-30, 3/-30 , 4/-60}
		{
			\node[std] (w\i) at (150-\i*60:2cm) {}; 	
				\path (w\i) ++(\j:5mm) node (Lw\i) {$w_{\i}$};		
			\draw[thick] (w0)--(w\i);        
		}	
	
		\draw[thick] (w1)--(w2)--(w3)--(w4)--(w5)--(w6)--(w1);
		
		\foreach \i / \j / \rad in {1/1.5/4,2/0/5,3/-1.5/4}
		{	
			\path (rcent) ++(90:\j cm) node [bblue] (v\i) {};
			\path (v\i) ++(0:\rad mm) node (Lv\i) {\color{blue} $v_{\i}$};
		}
	
		\draw[bline] (w1)--(v1)--(v2)--(v3)--(w4);
		\draw[bline] (w2)--(v1);
		
		\draw[bline] (v1) to[out=210,in=150] (v3);  	
		
		\draw[bline] (v3) to[out=230,in=290] (w5);  
		
		\draw[bline] (v2) to[out=-45, in=90]  (-15:4.9cm) to [out=270,in=0]    (-60:3.7cm) to[out=180,in=280] (w5); 
		\draw[bline] (v2) to[out=45, in=-90]  (15:4.9cm) to [out=-270,in=0]  (45:3.7cm) to[out=-180,in=20] (w1);	
		
		\path (w0) ++(60:12 mm) node [color=red] (X) {$X$};			
		\path (w0) ++(60:-12 mm) node [color=red] (Y) {$Y$};
		
		\path (w0) ++(0:13 mm) node [color=red] (A1) {$A_1$};	
		\path (w0) ++(-60:12 mm) node [color=red] (A2) {$A_2$};
		
		\path (w0) ++(-60:-12 mm) node [color=red] (B1) {$B_1$};	
		\path (w0) ++(180:12 mm) node [color=red] (B2) {$B_2$};
		
		\path (w2) ++(100:5 mm) node [color=red] (D1) {$D_1$};	
		\path (v1) ++(140:9 mm) node [color=red] (D2) {$D_2$};	
		\path (v2) ++(180:5 mm) node [color=red] (D3) {$D_3$};
		\path (v3) ++(250:9mm) node [color=red] (D4) {$D_4$};
		\path (w4) ++(-8:17 mm) node [color=red] (D5) {$D_5$};		
		
	\end{tikzpicture} 
	\caption{The graph $\cp{G_{3,3,6}}$ from Construction \ref{con:c:thet33k} such that $\ig{G_{3,3,6}} =  \ag{G_{3,3,6}}= \thet{3,3,6}$.}
	\label{fig:c:thet336}
\end{figure}	

\begin{figure}[H] \centering	
	\begin{tikzpicture}	[scale=0.8]
		\node (cent) at (0,0) {};
		\path (cent) ++(0:40 mm) node (rcent) {};	
		
		\node[std,label={0:$w_0$}] (w0) at (cent) {};
		
		\foreach \i in {1,3,4,5,6}
		{
			\node[std,label={150-\i*60: $w_{\i}$}] (w\i) at (150-\i*60:2cm) {}; 
			\draw[thick] (w0)--(w\i);        
		}
		\node[std,label={0: $w_{2}$}] (w2) at (30:2cm) {}; 
		\draw[thick] (w0)--(w2);        
		
		\draw[thick] (w1)--(w2)--(w3)--(w4)--(w5)--(w6)--(w1);
		
		\foreach \i / \j  / \k in {1/3/1, 2/2/2, 3/-1/{\ell-5}, 4/-2.2/{\ell-4}, 5/-3.4/{\ell-3}}
		{	
			\path (rcent) ++(90:\j cm) node [bblue] (v\i) {};
		}
		
		\path (v1) ++(45:4 mm) node (Lv1) {\color{blue} $v_{1}$};
		\path (v2) ++(200:5 mm) node (Lv2) {\color{blue} $v_{2}$};
		\path (v3) ++(150:6 mm) node (Lv3) {\color{blue} $v_{\ell-5}$};
		\path (v4) ++(0:8 mm) node (Lv4) {\color{blue} $v_{\ell-4}$};
		\path (v5) ++(260:5 mm) node (Lv5) {\color{blue} $v_{\ell-3}$};

		\draw[bline] (v1)--(v2);
		\draw[bline] (v3)--(v4)--(v5);
		
		\draw[bline,dashed] (v2)--(v3);
		
		\path (v4) ++(180:1.3 cm) coordinate  (c1) {};
		\draw[bline] (v3) to[out=190, in=90]  (c1) to [out=270,in=170] (v5);  
		
		\draw[bline] (w1)--(v1)--(w2);	
		
		\draw[bline] (v5) to[out=180,in=-45] (w4);  
		
		\draw[bline] (v5) to[out=200,in=290] (w5);  
		
		\draw[bline] (v4) to[out=-35, in=25]  (4.6cm,-4.2cm) to [out=205,in=270] (w5);  	

		\draw[bline] (w1) to[out=25, in=180]  (40:5.8cm) to [out=350,in=45]  (v2); 
		
		\draw[bline] (w1) to[out=38, in=180]  (40:6.4cm) to [out=350,in=45]  (v3); 
		
		\draw[bline] (w1) to[out=40, in=180]  (40:6.6cm) to [out=350,in=40]  (v4); 		
		
		\path (rcent) ++(90:1 cm) coordinate  (t1) {};
		\draw[bline, dashed] (w1) to[out=26, in=180]   (40:6cm)  to [out=350,in=45]  (t1); 
		
		\path (rcent) ++(90:0 cm) coordinate  (t2) {};
		\draw[bline, dashed] (w1) to[out=31, in=180]  (40:6.2cm) to [out=350,in=45]  (t2); 
		
		\path (w0) ++(60:12 mm) node [color=red] (X) {$X$};			
		\path (w0) ++(60:-12 mm) node [color=red] (Y) {$Y$};
		
		\path (w0) ++(0:14 mm) node [color=red] (A1) {$A_1$};	
		\path (w0) ++(-60:12 mm) node [color=red] (A2) {$A_2$};
		
		\path (w0) ++(-60:-12 mm) node [color=red] (B1) {$B_1$};	
		\path (w0) ++(180:12 mm) node [color=red] (B2) {$B_2$};
		
		\path (w2) ++(105:8 mm) node [color=red] (D1) {$D_1$};	
		\path (v1) ++(170:9 mm) node [color=red] (D2) {$D_2$};	
		\path (v4) ++(180:7 mm) node [color=red] (D3) {$D_{\ell-3}$};
		\path (w4) ++(-55:25mm) node [color=red] (D4) {$D_{\ell-2}$};
		\path (w4) ++(-50:15 mm) node [color=red] (D5) {$D_{\ell-1}$};		
		
	\end{tikzpicture} 
	\caption{The graph $\cp{G_{3,3,\ell}}$ from Construction \ref{con:c:thet33k} such that $\ig{G} = \ag{G} = \thet{3,3,\ell}$ for $\ell \geq 6$.}
	\label{fig:c:thet33k}
\end{figure}	

	\begin{lemma}\label{lem:c:thet33k}
		If $\cp{G_{3,3,\ell}}$ is the graph constructed by Construction \ref{con:c:thet33k}, then $\ig{G_{3,3,\ell}} = \ag{G_{3,3,\ell}} = \thet{3,3,\ell}$ for $\ell \geq 6$.
	\end{lemma}

	\subsubsection{Construction of $\cp{G}$ for $\thet{3,4,4}$ }  \label{subsubsec:c:344}

	\begin{cons}\label{con:c:thet344} See Figure \ref{fig:c:thet344}.	
		Begin with a copy of the graph $\cp{G}$ from Construction \ref{con:c:thet244}   for $\thet{2,4,4}$, which we rename here as $\cp{G_{2,4,4}}$.  Subdivide the edge $w_2w_3$, adding a new vertex $u_1$, and joining $u_1$ to $w_0$.  
		Delete the vertex $z$.   Call this graph $\cp{G_{3,4,4}}$.
		

	\end{cons}

	\begin{figure}[H] \centering	
		\begin{tikzpicture}	[scale=0.8]
			
			\node (cent) at (0,0) {};
			\path (cent) ++(0:50mm) node (rcent) {};	
			
			\node[std] (w0) at (cent) {};
			\path (w0) ++(55:6 mm) node (Lw0) {$w_0$};	
			
			\foreach \i in {1,2,3,4,5,6}
			{
				\node[std,label={150-\i*60: $w_{\i}$}] (w\i) at (150-\i*60:2cm) {}; 
				\draw[thick] (w0)--(w\i);        
			}
			
			\node[std,label={0: $u_1$}] (w) at (0:2cm) {}; 
			\draw[thick] (w0)--(w);
			
			\draw[thick] (w1)--(w2)--(w)--(w3)--(w4)--(w5)--(w6)--(w1);
			
			
			\path (rcent) ++(90:0 cm) node [bblue] (v) {};
			\path (v) ++(0:4 mm) node (Lv) {\color{blue} $v$};	
			
			\foreach \i / \j in {2,3}
			\draw[bline] (v)--(w\i);	
			
			\draw[bline] (w1)   to[out=0,in=100] (v); 	
			\draw[bline] (w4)   to[out=0,in=-100] (v); 
			\draw[bline] (w4) to[out=-15, in=270]  (0:6.5cm)  to[out=90,in=15] (w1);

			\path (w0) ++(180:4 cm) node [bzed, label={[zelim]180:$\boldsymbol{z'}$}] (z1) {};
			\draw[zline] (w0)   to[out=165,in=15] (z1); 
			\draw[zline] (w1)   to[out=180,in=80] (z1); 		
			\draw[zline] (w4)   to[out=180,in=-80] (z1);

			\path (w0) ++(60:12 mm) node [color=red] (X) {$X$};	
			\path (w0) ++(15:12 mm) node [color=red] (A1) {$A_1$};	
			\path (w0) ++(-15:12 mm) node [color=red] (A2) {$A_2$};	
			\path (w0) ++(-60:12 mm) node [color=red] (Y) {$Y$};	
			
			\foreach \i/\c in {1/120,2/180,3/240} {				
				\path (w0) ++(\c:12 mm) node [color=red] (B\i) {$B_{\i}$};	
			}
			
			\path (w2) ++(05:15 mm) node [color=red] (D1) {$D_1$};	
			\path (v) ++(60:8 mm) node [color=red] (D2) {$D_2$};	
			\path (w3) ++(-05:15 mm) node [color=red] (D3) {$D_3$};
			
		\end{tikzpicture} 
		\caption{The   graph  \; $\cp{G_{3,4,4}}$ \;  from    Construction   \ref{con:c:thet344}    such    that    $\ig{G_{3,4,4}}$    $=$  $\thet{3,4,4}$.}
		\label{fig:c:thet344}
	\end{figure}		

	\begin{lemma}\label{lem:c:thet344}
		If $\cp{G_{3,4,4}}$ is the graph constructed by Construction \ref{con:c:thet344}, then $\ig{G_{3,4,4}} = \thet{3,4,4}$.
	\end{lemma}

	\subsubsection{Construction of $\cp{G}$ for $\thet{3,4,\ell}$, $\ell \geq 5$}  \label{subsubsec:c:34l}


	\begin{cons}\label{cons:c:thet34k} 
		Begin with a copy of the graph $\cp{G_{3,4,4}}$ from Construction \ref{con:c:thet344} for $\thet{3,4,4}$.  Subdivide the edge $w_1w_6$ $\ell-4$ times (for $\ell \geq 5$), adding the new vertices $w_7$, $w_8, \dots, w_{\ell+2}$.  Join $w_7, w_8, \dots, w_{\ell+2}$ to $w_0$, so that $w_0,w_1,\dots,w_{\ell+2}$ forms a wheel.  Call this graph $\cp{G_{3,4,\ell}}$. 
	\end{cons}
	
	\begin{lemma}\label{lem:c:thet34k}
		If $\cp{G_{3,4,\ell}}$ is the graph constructed by Construction \ref{cons:c:thet34k}, then $\ig{G_{3,4,\ell}} = \thet{3,4,\ell}$ for $\ell \geq 5$.
	\end{lemma}
	
	\subsubsection{Construction of $\cp{G}$ for $\thet{3,5,5}$}  \label{subsubsec:c:355}


	\begin{cons} \label{con:c:thet355}
		Refer to  Figure \ref{fig:c:thet355}.	
		Begin with a copy of the graph $\cp{G_{3,3,5}}$ from  Construction \ref{con:c:thet335}.
		Subdivide the edge $w_1w_6$ twice, adding the new vertices $w_7$ and  $w_8$.  Join $w_7$ and $w_8$ to $w_0$, so that $w_0,w_1,\dots,w_{8}$ forms a wheel.  Call this graph $\cp{G_{3,5,5}}$.
	\end{cons}

	\begin{figure}[H] \centering	
		\begin{tikzpicture}	[scale=0.8]
			\node (cent) at (0,0) {};
			\path (cent) ++(0:40 mm) node (rcent) {};	
			
			\node[std,label={0:$w_0$}] (w0) at (cent) {};
			
			\foreach \i in {1,2,3,4,5}
			{
				\node[std,label={150-\i*60: $w_{\i}$}] (w\i) at (150-\i*60:2cm) {}; 
				\draw[thick] (w0)--(w\i);        
			}
			
			\foreach \i in {6,7,8}
			{
				\node[std,label={360-\i*30: $w_{\i}$}] (w\i) at (360-\i*30:2cm) {}; 
				\draw[thick] (w0)--(w\i);   
			}
			
			\draw[thick] (w1)--(w2)--(w3)--(w4)--(w5)--(w6)--(w7)--(w8)--(w1);
			
			\foreach \i/\j in {1/1.7, 2/-0.5}
			{
				\path (rcent) ++(90:\j cm) node [bblue,label={[blue]0: $v_\i$}] (v\i) {};
			}	
			
			\draw[bline] (w1)--(v1)--(v2)--(w2)--(v1);
			
			\draw[bline] (v2) to[out=230,in=0] (w4);  	
			
			\draw[bline] (v2) to [out=260,in=15]  (-60:3.5cm) to [out=195,in=270] (w5);  
			
			\draw[bline] (v1) to [out=-45, in=70]  (-15:5cm) to [out=250,in=15]  (-60:4cm) to[out=195,in=260] (w5); 
			
			\path (w0) ++(60:12 mm) node [color=red] (X) {$X$};			
			\path (w0) ++(60:-12 mm) node [color=red] (Y) {$Y$};
			
			\path (w0) ++(0:13 mm) node [color=red] (A1) {$A_1$};	
			\path (w0) ++(-60:12 mm) node [color=red] (A2) {$A_2$};
			
			\path (w0) ++(105:12 mm) node [color=red] (B1) {$B_1$};	
			\path (w0) ++(135:12 mm) node [color=red] (B2) {$B_2$};
			\path (w0) ++(165:12 mm) node [color=red] (B2) {$B_3$};
			\path (w0) ++(195:12 mm) node [color=red] (B2) {$B_4$};
			
			\path (w2) ++(100:5 mm) node [color=red] (D1) {$D_1$};	
			\path (w2) ++(-10:15mm) node [color=red] (D2) {$D_2$};	
			\path (v2) ++(280:10 mm) node [color=red] (D3) {$D_3$};
			\path (w4) ++(-25:15 mm) node [color=red] (D4) {$D_4$};	
			
		\end{tikzpicture} 
		\caption{A graph $\cp{G_{3,5,5}}$ from Construction \ref{con:c:thet355} such that $\ig{G_{3,5,5}} = \thet{3,5,5}$.}
		\label{fig:c:thet355}
	\end{figure}		

	\begin{lemma}\label{lem:c:thet355}
		If $\cp{G_{3,5,5}}$ is the graph constructed by Construction \ref{con:c:thet355}, then $\ig{G_{3,5,5}} = \ag{G_{3,5,5}} =\thet{3,5,5}$.
	\end{lemma}

	Notice that subdividing only once in Construction \ref{con:c:thet355} (adding only $w_7$ and not $w_8$) gives an alternative  (planar) construction for $\thet{3,4,5}$.

	\subsection{$\thet{j,k,\ell}$ for $4 \leq j \leq k \leq \ell$ and $3 \leq j \leq k \leq \ell$, $\ell \geq 6$}  \label{subsec:c:jkl}

	\subsubsection{Construction of $\cp{G}$ for $\thet{j,k,\ell}$ for   $4 \leq j \leq k \leq \ell \leq 5$}  \label{subsubsec:c:4kl}

	\begin{cons} \label{con:c:thet444}
		Refer to  Figure \ref{fig:c:thet444}.	
		Begin with a copy of the graph $\cp{G_{3,4,4}}$ from  Construction \ref{con:c:thet344}.
		Subdivide the edge $u_1w_3$, adding the new vertex $u_2$.  Join $u_2$ to $w_0$, so that
		\begin{align*}
			w_0,w_1,w_2,u_1,u_2,w_3,\dots,w_{6}
		\end{align*}
		forms a wheel.  Call this graph $\cp{G_{4,4,4}}$.
	\end{cons}

	\begin{figure}[H] \centering	
		\begin{tikzpicture}	[scale=0.8]
			
			\node (cent) at (0,0) {};
			\path (cent) ++(0:50mm) node (rcent) {};	
			
			\node[std] (w0) at (cent) {};
			\path (w0) ++(55:6 mm) node (Lw0) {$w_0$};	
			
			\foreach \i in {1,2,3,4,5,6}
			{
				\node[std,label={150-\i*60: $w_{\i}$}] (w\i) at (150-\i*60:2cm) {}; 
				\draw[thick] (w0)--(w\i);        
			}
			
			\node[std,label={0: $u_1$}] (u1) at (10:2cm) {}; 
			\draw[thick] (w0)--(u1);

			\node[std,label={0: $u_2$}] (u2) at (-10:2cm) {}; 
			\draw[thick] (w0)--(u2);

			\draw[thick] (w1)--(w2)--(u1)--(u2)--(w3)--(w4)--(w5)--(w6)--(w1);
			
			
			\path (rcent) ++(90:0 cm) node [bblue] (v) {};
			\path (v) ++(0:4 mm) node (Lv) {\color{blue} $v$};	
			
			\foreach \i / \j in {2,3}
			\draw[bline] (v)--(w\i);	
			
			\draw[bline] (w1)   to[out=0,in=100] (v); 	
			\draw[bline] (w4)   to[out=0,in=-100] (v); 
			\draw[bline] (w4) to[out=-15, in=270]  (0:6.5cm)  to[out=90,in=15] (w1);

			\path (w0) ++(180:4 cm) node [bzed, label={[zelim]180:$\boldsymbol{z'}$}] (z1) {};
			\draw[zline] (w0)   to[out=165,in=15] (z1); 
			\draw[zline] (w1)   to[out=180,in=80] (z1); 		
			\draw[zline] (w4)   to[out=180,in=-80] (z1);

			\path (w0) ++(60:12 mm) node [color=red] (X) {$X$};	
			\path (w0) ++(20:15 mm) node [color=red] (A1) {$A_1$};	
			\path (w0) ++(0:15 mm) node [color=red] (A2) {$A_2$};	
			\path (w0) ++(-20:15 mm) node [color=red] (A3) {$A_3$};	
			\path (w0) ++(-60:12 mm) node [color=red] (Y) {$Y$};	
			
			\foreach \i/\c in {1/120,2/180,3/240} {				
				\path (w0) ++(\c:12 mm) node [color=red] (B\i) {$B_{\i}$};	
			}
			
			\path (w2) ++(05:15 mm) node [color=red] (D1) {$D_1$};	
			\path (v) ++(60:8 mm) node [color=red] (D2) {$D_2$};	
			\path (w3) ++(-05:15 mm) node [color=red] (D3) {$D_3$};
			
		\end{tikzpicture} 
		\caption{The graph \; $\cp{G_{4,4,4}}$ \; from Construction \ref{con:c:thet444} such that $\ig{G_{4,4,4}}$ $=$ $\thet{4,4,4}$.}
		\label{fig:c:thet444}
	\end{figure}		

	\begin{lemma}\label{lem:c:thet444}
		If $\cp{G_{4,4,4}}$ is the graph constructed by Construction \ref{con:c:thet444}, then $\ig{G_{4,4,4}} = \thet{4,4,4}$.
	\end{lemma}

	\begin{cons} \label{con:c:thetjk5}
		Begin with a copy of the graph $\cp{G_{3,3,5}}$ from  Construction \ref{con:c:thet335}.
		For $k=4$ subdivide the edge $w_1w_6$ once, adding the vertex $w_7$; for $k=5$, subdivide a second time, adding the vertex $w_8$. 
		For $j=4$, subdivide the edge $w_2w_3$, adding the vertex $u_1$; for $j=5$ ($j \leq k$), subdivide a second time, adding the vertex $u_2$.  Connect all new vertices to $w_0$ to form a wheel.
		Call this graph $\cp{G_{j,k,5}}$ for $4 \leq j \leq k \leq 5$.
	\end{cons}
	
	An example of the construction of $\cp{G_{5,5,5}}$ is given in Figure \ref{fig:c:thet555} below.

\begin{figure}[H] \centering	
	\begin{tikzpicture}	[scale=0.8]
		\node (cent) at (0,0) {};
		\path (cent) ++(0:40 mm) node (rcent) {};	
		
		\node[std] (w0) at (cent) {};
		\path (w0) ++(60: 8 mm) node (Lw0) {${w_0}$};	
		
		\foreach \i in {1,4,5}
		{
			\node[std,label={150-\i*60: $w_{\i}$}] (w\i) at (150-\i*60:2cm) {}; 
			\draw[thick] (w0)--(w\i);        
		}
	
		\node[std] (w2) at (30:2cm) {}; 
				\path (w2) ++(39:8mm) node (Lw2) {$w_2$};	
		\draw[thick] (w0)--(w2); 
		
		\foreach \i in {6,7,8}
		{
			\node[std,label={360-\i*30: $w_{\i}$}] (w\i) at (360-\i*30:2cm) {}; 
			\draw[thick] (w0)--(w\i);   
		}
		
		\foreach \i in {1,2}
		{
			\node[std,label={0: $u_{\i}$}] (u\i) at (30-\i*30:2cm) {}; 
			\draw[thick] (w0)--(u\i);   
		}		
		
		\node[std] (w3) at (300:2cm) {}; 
				\path (w3) ++(5:6 mm) node (Lw3) {$w_3$};	
		\draw[thick] (w0)--(w3);

		\draw[thick] (w1)--(w2)--(u1)--(u2)--(w3)--(w4)--(w5)--(w6)--(w7)--(w8)--(w1);
		
		\foreach \i/\j in {1/1.8, 2/-0.5}
		{
			\path (rcent) ++(90:\j cm) node [bblue,label={[blue]0: $v_\i$}] (v\i) {};
		}	
		
		\draw[bline] (w1)--(v1)--(v2)--(w2)--(v1);
		
		\draw[bline] (v2) to[out=230,in=0] (w4);  	
		
		\draw[bline] (v2) to [out=260,in=15]  (-60:3.5cm) to [out=195,in=270] (w5);  
		
		\draw[bline] (v1) to [out=-45, in=70]  (-15:5cm) to [out=250,in=15]  (-60:4cm) to[out=195,in=260] (w5); 
		
		\path (w0) ++(60:13 mm) node [color=red] (X) {$X$};			
		\path (w0) ++(60:-12 mm) node [color=red] (Y) {$Y$};
		
		\path (w0) ++(15:15 mm) node [color=red] (A1) {$A_1$};	
		\path (w0) ++(-15:15 mm) node [color=red] (A2) {$A_2$};
		\path (w0) ++(-45:15 mm) node [color=red] (A1) {$A_3$};	
		\path (w0) ++(-75:15 mm) node [color=red] (A2) {$A_4$};
		
		\path (w0) ++(105:15 mm) node [color=red] (B1) {$B_1$};	
		\path (w0) ++(135:15 mm) node [color=red] (B2) {$B_2$};
		\path (w0) ++(165:15 mm) node [color=red] (B2) {$B_3$};
		\path (w0) ++(195:15 mm) node [color=red] (B2) {$B_4$};
		
		\path (w2) ++(100:5 mm) node [color=red] (D1) {$D_1$};	
		\path (w2) ++(-10:15mm) node [color=red] (D2) {$D_2$};	
		\path (v2) ++(280:10 mm) node [color=red] (D3) {$D_3$};
		\path (w4) ++(-25:15 mm) node [color=red] (D4) {$D_4$};	
		
	\end{tikzpicture} 
	\caption{A graph $\cp{G_{5,5,5}}$ from Construction \ref{con:c:thet355} such that $\ig{G_{5,5,5}} = \thet{5,5,5}$.}
	\label{fig:c:thet555}
\end{figure}		

	\begin{lemma}\label{lem:c:thetjk5}
		If $\cp{G_{j,k,5}}$ is the graph constructed by Construction \ref{con:c:thetjk5}, then $\ig{G_{j,k,5}} = \ag{G_{j,k,5}} =\thet{j,k,5}$ for $4 \leq j \leq k \leq 5$.
	\end{lemma}

	\subsubsection{Construction of $\cp{G}$ for $\thet{j,k,\ell}$ for   $3 \leq j \leq k \leq \ell$, $\ell \geq 6$.}  \label{subsubsec:c:jkl}

	\begin{cons} \label{con:c:thetijk}
		Begin with a copy of the graph $\cp{G_{3,3,\ell}}$ from  Construction \ref{con:c:thet33k} for $\thet{3,5,\ell}$ for $\ell \geq 6$.
		For $3 \leq k \leq \ell$, subdivide the edge $w_1w_6$ $k-3$ times, adding the new vertices $w_7, w_8, \dots, w_{k+3}$.  Join each of  $w_7, w_8, \dots, w_{k+3}$  to $w_0$, so that $w_0,w_1,\dots,w_{k+3}$ forms a wheel. 
		Then, for  $3 \leq j \leq k$,  subdivide the edge $w_2w_3$ $j-3$ times, adding the new vertices $u_1, u_2, \dots, u_{j-3}$.  Again, join each  of $u_1, u_2, \dots, u_{j-3}$ to $w_0$ to form a wheel.
		Call this graph $\cp{G_{j,k,\ell}}$ for $3 \leq j \leq k \leq \ell$ and $\ell \geq 6$.
	\end{cons}

\begin{figure}[H] \centering	
	\begin{tikzpicture}	[scale=0.8]
		\node (cent) at (0,0) {};
		\path (cent) ++(0:45 mm) node (rcent) {};	
		
		\node[std] (w0) at (cent) {};
		\path (w0) ++(55: 7 mm) node (Lw0) {${w_0}$};	
		
		\foreach \i / \j / \r in {1/1/2,  {4b}/{3}/7,  {5b}/{6}/11, {6b}/{k+3}/13}
		{
			\node[std,label={150-\r*30: $w_{\j}$}] (w\i) at (150-\r*30:2cm) {}; 
			\draw[thick] (w0)--(w\i);        
		}
		
		\node[std,label={270: $w_{4}$}] (w4) at (260:2cm) {}; 
		\draw[thick] (w0)--(w4);        	
	
		\node[std,label={0: $w_{2}$}] (w2) at (30:2cm) {}; 
		\draw[thick] (w0)--(w2);       	
		
		\node[std,label={210: $w_{5}$}] (w5) at (220:2cm) {}; 
		\draw[thick] (w0)--(w5);

		\node[std,label={0: $u_1$}] (w3) at (0:2cm) {}; 
		\draw[thick] (w0)--(w3);     
		
		\node[std] (u2) at (-15:2cm) {}; 
			\path (u2) ++(-8: 6 mm) node (Lu2) {${u_2}$};	
		\draw[dashed] (w0)--(u2);        
		
		\node[std] (uj) at (-45:2cm) {};
			\path (uj) ++(-25: 8 mm) node (Luj) {$u_{j-3}$};	
		\draw[dashed] (w0)--(uj);

		\node[std] (q1) at (165:2cm) {}; 
			\path (q1) ++(170: 6 mm) node (Lq1) {${w_7}$};	
		\draw[dashed] (w0)--(q1);        
		
		\node[std] (qk) at (135:2cm) {};
				\path (qk) ++(160: 8 mm) node (Lqk) {$w_{k+2}$};	
		\draw[dashed] (w0)--(qk);

		\draw[thick] (w5)--(w5b)--(q1)--(qk)--(w6b)--(w1);
		
		\draw[thick] (w2)--(w3)--(u2)--(uj)--(w4b)--(w4);

		\draw[thick] (w1)--(w2);
		\draw[thick] (w4)--(w5);
		
		\foreach \i / \j  / \k in {1/3/1, 2/2/2, 3/-1/{\ell-5}, 4/-2.2/{\ell-4}, 5/-3.4/{\ell-3}}
		{	
			\path (rcent) ++(90:\j cm) node [bblue] (v\i) {};
		}
	
		\path (v1) ++(90:4 mm) node (Lv1) {\color{blue} $v_{1}$};
		\path (v2) ++(200:5 mm) node (Lv2) {\color{blue} $v_{2}$};
		\path (v3) ++(150:6 mm) node (Lv3) {\color{blue} $v_{\ell-5}$};
		\path (v4) ++(0:8 mm) node (Lv4) {\color{blue} $v_{\ell-4}$};
		\path (v5) ++(260:5 mm) node (Lv5) {\color{blue} $v_{\ell-3}$};

		\draw[bline] (v1)--(v2);
		\draw[bline] (v3)--(v4)--(v5);
		
		\draw[bline,dashed] (v2)--(v3);
		
		\path (v4) ++(180:1.3 cm) coordinate  (c1) {};
		\draw[bline] (v3) to[out=190, in=90]  (c1) to [out=270,in=170] (v5);  
		
		\draw[bline] (w1)--(v1)--(w2);	
		
		\draw[bline] (v5) to[out=180,in=-45] (w4);  
		
		\draw[bline] (v5) to[out=200,in=290] (w5);

		\draw[bline] (v4) to[out=-35, in=25]  (4.6cm,-4.2cm) to [out=205,in=270] (w5);  	

		\draw[bline] (w1) to[out=25, in=180]  (40:5.8cm) to [out=350,in=45]  (v2); 
		
		\draw[bline] (w1) to[out=38, in=180]  (40:6.4cm) to [out=350,in=45]  (v3); 
		
		\draw[bline] (w1) to[out=40, in=180]  (40:6.6cm) to [out=350,in=40]  (v4); 		
		
		\path (rcent) ++(90:1 cm) coordinate  (t1) {};
		\draw[bline, dashed] (w1) to[out=26, in=180]   (40:6cm)  to [out=350,in=45]  (t1); 
		
		\path (rcent) ++(90:0 cm) coordinate  (t2) {};
		\draw[bline, dashed] (w1) to[out=31, in=180]  (40:6.2cm) to [out=350,in=45]  (t2); 
		
		\path (w0) ++(60:13 mm) node [color=red] (X) {$X$};			
		\path (w0) ++(60:-14 mm) node [color=red] (Y) {\scalebox{0.9}{$Y$}};
		
		\path (w0) ++(15:14 mm) node [color=red] (A1)  {\scalebox{0.9}{$A_1$}};	
		\path (w0) ++(-80:15 mm) node [color=red] (A2) {\scalebox{0.9}{$A_{j-1}$}};
		
		\path (w0) ++(-75:-14 mm) node [color=red] (B1) {\scalebox{0.9}{$B_1$}};	
		\path (w0) ++(200:13 mm) node [color=red] (B2) {\scalebox{0.9}{$B_{k-1}$}};

		\path (w2) ++(105:8 mm) node [color=red] (D1) {$D_1$};	
		\path (v1) ++(170:10 mm) node [color=red] (D2) {$D_2$};	
		\path (v4) ++(180:7 mm) node [color=red] (D3) {$D_{\ell-3}$};
		\path (w4) ++(-46:30mm) node [color=red] (D4) {$D_{\ell-2}$};
		\path (w4) ++(-50:15 mm) node [color=red] (D5) {$D_{\ell-1}$};

		\coordinate  (c1) at ($(u2)!0.33!(uj)$) {};
		\coordinate  (c2) at ($(u2)!0.66!(uj)$) {};
		
		\draw [dashed] (c2)--(w0);
		\draw [dashed] (c1)--(w0);		
		
		\coordinate  (d1) at ($(q1)!0.33!(qk)$) {};
		\coordinate  (d2) at ($(q1)!0.66!(qk)$) {};
		
		\draw [dashed] (d2)--(w0);
		\draw [dashed] (d1)--(w0);

	\end{tikzpicture} 
	\caption{The graph $\cp{G_{j,k,\ell}}$ from Construction \ref{con:c:thetijk} such that $\ig{G_{j,k,\ell}} = \thet{j,k,\ell}$ for $3\leq j \leq k \leq \ell$ and $\ell \geq 6$.}
	\label{fig:c:thetijk}
\end{figure}	

	\begin{lemma}\label{lem:c:thetjkl}
		If  $\cp{G_{j,k,\ell}}$ for $3 \leq j \leq k \leq \ell$ and $\ell \geq 6$ is the graph constructed by Construction \ref{con:c:thetijk}, then $\ig{G_{j,k,\ell}} = \ag{G_{j,k,\ell}} = \thet{j,k,\ell}$.
	\end{lemma}


	The lemmas above imply the sufficiency of Theorem \ref{thm:c:thetas}: if a theta graph is not one of seven exceptions listed, then it is an $i$-graph.  In the next section, we complete the proof by examining the exception cases.

	

	\section{Theta Graphs that are not $i$-Graphs} \label{sec:c:nonthetas}

	In this section we show that $\thet{2,2,4}$, $\thet{2,3,3}$, $\thet{2,3,4}$, and $\thet{3,3,3}$ are not $i$-graphs; together with Proposition \ref{prop:i:diamond}, this completes the proof of Theorem \ref{thm:c:thetas}.

	\subsection{$\thet{2,2,4}$ is not an $i$-Graph}  \label{subsubsec:c:224}
	
	\begin{prop}	\label{prop:c:thet224}
		The graph $\thet{2,2,4}$ is not $i$-graph realizable.
	\end{prop}
	
	\begin{figure}[H]
		\centering
		\begin{tikzpicture}
			
			\node(cent) at (0,0) {};			
			\path (cent) ++(0:60 mm) node(rcent)  {};		
			
			\foreach \i/\j/\c in {x/X/90,y/Y/-90} {				
				\path (cent) ++(\c:20 mm) node [std] (v\i) {};
			}
			
			\path (vx) ++(90:5 mm) node (Lvx) {$\color{blue}X=\{x_1,x_2,x_3,\dots, x_k\}$};
			\path (vy) ++(-90:5 mm) node (Lvy) {$\color{blue}Y=\{y_1,y_2,x_3,\dots, x_k\}$};
			
			\foreach \i/\j/\c in {a/A/180,b/B/0} {				
				\path (cent) ++(\c:8 mm) node [std] (v\i) {};
			}

			\path (va) ++(175:20 mm) node (Lvau) {$\color{blue}A=$};
			\path (va) ++(185:20 mm) node (Lvad) {$\color{blue}\{y_1,x_2,x_3,\dots, x_k\}$};
			\path (vb) ++(0:25 mm) node (Lvb) {\color{blue}$B=\{x_1,y_2,x_3,\dots, x_k\}$};

			\draw[thick] (vx)--(va)--(vy)--(vb)--(vx)--cycle;				
			
			\foreach \i/\c in {1/15,2/0,3/-15} {				
				\path (rcent) ++(90:\c mm) node [std] (c\i) {};
			}
			
			\draw[thick] (vx)--(c1)--(c2)--(c3)--(vy);
			
			
			\path (c1) ++(0:25 mm) node (Lvc1) {\color{blue}$C_1=\{x_1,x_2,y_3,\dots, x_k\}$};
			\path (c2) ++(0:8 mm) node (Lvc2) {\color{red}$C_2=$};
			\path (Lvc2) ++(0:20 mm) node (Lvc2C) {\color{red} OR };
			\path (Lvc2C) ++(90:5 mm) node (Lvc2A) {\color{red} $\{y_1,x_2,y_3,\dots, x_k\}$};
			\path (Lvc2C) ++(-90:5 mm) node (Lvc2B) {\color{red} $\{x_1,y_2,y_3,\dots, x_k\}$};
			
			\path (c3) ++(0:5 mm) node (Lvc3) {\color{red}$C_3$};

		\end{tikzpicture}				
		\caption{$H=\thet{2,2,4}$ non-construction.}
		\label{fig:c:thet224}%
	\end{figure}

	\begin{proof}
		Suppose to the contrary that $\thet{2,2,4}$ is realizable as an $i$-graph, and  that $H = \thet{2,2,4} \cong \ig{G}$ for some graph $G$.  Label the vertices of $H$ as in Figure \ref{fig:c:thet224}.
		
		From Proposition \ref{fig:i:C4struct}, the composition of the following $i$-sets of $G$ are immediate: 
		
		\begin{align*}
		X&=\{x_1,x_2,x_3,\dots,x_k\},& Y=\{y_1,y_2,x_3,\dots,x_k\}, \\
		A&=\{y_1,x_2,x_3,\dots,x_k\},& B=\{x_1,y_2,x_3,\dots,x_k\}, \\
		C_1&=\{x_1,x_2,y_3,\dots,x_k\} 
		\end{align*}
		
		\noindent
		where $k \geq 3$ and $y_1, y_2, y_3$ are three distinct vertices in $G-X$.  These sets are illustrated in blue in Figure \ref{fig:c:thet224}.  This leaves only the composition of $C_2$ and $C_3$ (in red) to be determined.  
		As we construct $C_2$, notice first that $y_3\in C_2$; otherwise, if say some other $z \in C_2$ so that $\edge{C_1,{y_3},z,C_2}$, then $\edge{X,x_3,z,C_2}$, and $XC_2 \in E(H)$.  Thus, a token on one of $\{x_1,x_2,x_4,\dots,x_k\}$ moves in the transition from $C_1$ to $C_2$.  We consider three cases.
		
		\textbf{Case 1:}  The token on $x_1$ moves.  If $\edge{C_1,{x_1},z,C_2}$ for some $z \notin \{y_1,y_2\}$, then $|C_2 \cap Y |=3$, contradicting the distance requirement between $i$-sets from Observation \ref{obs:i:dist}.    Moreover, from the composition of  $B$, $x_1 \not\sim y_2$, and so  $\edge{C_1,{x_1},y_1,C_2}$, so that $C_2 = \{y_1,x_2,y_3,\dots, x_k\}$.  However, since $x_3 \sim y_3$, we have that $\edge{A,x_3,y_3,C_2}$, so that $AC_2 \in E(G)$, a contradiction.
		
		
		\textbf{Case 2:}  The token on $x_2$ moves.  An argument similar to Case 1 constructs $C_2 = \{x_1,y_2,y_3,\dots,x_k\}$, with $\edge{B,x_3,y_3,C_2}$, resulting in the contradiction $BC_2\in E(G)$.

		\textbf{Case 3:}  The token on $x_i$ for some $i\in \{4,5,\dots,k\}$ moves.	From the compositions of $X$ and $Y$, $x_i$ is not adjacent to any of $\{x_3,y_1,y_2\}$, so the token at $x_i$ moves to some other vertex, say $z$, so that $\edge{C_1,x_i,z,C_2}$ and $\{x_1,x_2,y_3,z\} \subseteq C_2$.  This again contradicts the distance requirement of Observation \ref{obs:i:dist} as $|C_2 \cap Y| = 4$.

		\noindent In all cases, we fail to construct a graph $G$ with $\ig{G} \cong \thet{2,2,4}$ and so conclude that no such graph exists.		
	\end{proof}
	

	\subsection{ $\thet{2,3,3}$ is not an $i$-Graph} \label{subsubsec:c:233}

	\begin{prop}	\label{prop:c:thet233}
		The graph $\thet{2,3,3}$ is not $i$-graph realizable.
	\end{prop}
	
	\begin{figure}[H]
		\centering
		\begin{tikzpicture}
			
			\node(cent) at (0,0) {};			
			\path (cent) ++(0:60 mm) node(rcent)  {};		
			
			\foreach \i/\j/\c in {x/X/90,y/Y/-90} {				
				\path (cent) ++(\c:20 mm) node [std] (v\i) {};
			}
			
			\path (vx) ++(90:5 mm) node (Lvx) {$\color{blue}X=\{x_1,x_2,x_3,\dots, x_k\}$};
			\path (vy) ++(-90:5 mm) node (Lvy) {$\color{red}Y$};
			
			\path (cent) ++(180:8 mm) node [std] (va) {};	
			
			\path (va) ++(175:20 mm) node (Lvau) {$\color{blue}A=$};
			\path (va) ++(185:20 mm) node (Lvad) {$\color{blue}\{y_1,x_2,x_3,\dots, x_k\}$};
			
			\draw[thick] (vx)--(va)--(vy);

			
			\path (cent) ++(0:8 mm) node (vb) {};
			
			\path (vb) ++(90:8 mm) node [std] (vb1) {};
			\path (vb1) ++(0:25 mm) node (Lvb) {\color{blue}$B_1=\{x_1,y_2,x_3,\dots, x_k\}$};
			
			\path (vb) ++(90:-8 mm) node [std] (vb2) {};
			\path (vb2) ++(0:5 mm) node (Lvb) {\color{red}$B_2$};	
			
			\draw[thick](vx)--(vb1)--(vb2)--(vy);

			\foreach \i/\c in {1/14,2/-14} {				
				\path (rcent) ++(90:\c mm) node [std] (c\i) {};
			}
			
			\draw[thick] (vx)--(c1)--(c2)--(vy);
			
			
			\path (c1) ++(0:25 mm) node (Lvc1) {\color{blue}$C_1=\{x_1,x_2,y_3,\dots, x_k\}$};
			\path (c2) ++(0:5 mm) node (Lvc2) {\color{red}$C_2$};
			

		\end{tikzpicture}				
		\caption{$H=\thet{2,3,3}$ non-construction.}
		\label{fig:c:thet233}%
	\end{figure}

	\begin{proof}
		To begin, we proceed similarly to the proof of Proposition \ref{prop:c:thet224}: suppose to the contrary that $\thet{2,3,3}$ is realizable as an $i$-graph, and  that $H = \thet{2,3,3} \cong \ig{G}$ for some graph $G$.  Label the vertices of $H$ as in Figure \ref{fig:c:thet233}.  As before, the corresponding $i$-sets in blue are established from previous results, and those in red are yet to be determined.
		Moreover, from the composition of these four blue $i$-sets, we observe that for each $i\in \{1,2,3\}$, $x_i \sim y_j$ if and only if $i=j$.

		Unlike the construction for $\thet{2,2,4}$, we no longer start with knowledge of the exact composition of $Y$.  	
		We proceed with a series of observations on the contents of the various $i$-sets:
		
		\begin{enumerate}[itemsep=1pt, label=(\roman*)]
			\item $y_1 \in Y$, $y_2\in B_2$, and $y_3 \in C_2$  by three applications of  Proposition \ref{prop:i:C4struct}.  \label{prop:c:thet233:a}
			
			\item $y_1 \not\in B_2$ and $y_1 \not\in C_2$.  If $y_1 \in B_2$, then $\edge{B_1,x_1,y_1,B_2}$ (because $A$ shows that $y_1$ is not adjacent to $x_3,\dots,x_k$) so that $B_2  = \{y_1,y_2,x_3,\dots,x_k\}$, and therefore $\edge{A,x_2,y_2,B_2}$, which is impossible.   Similarly, if $y_1 \in C_2$, then $\edge{A, x_3,y_3,C_2}$, which is also impossible.  \label{prop:c:thet233:b}
			
			\item $y_3 \notin B_2$.  Otherwise, $B_2 = \{x_1,y_2,y_3,x_4,\dots,x_k\}$ and so $\edge{C_1,x_2,y_2,B_2}$.  \label{prop:c:thet233:c}
			
			\item $y_2 \not\in Y$ and $y_3\notin Y$.  If $y_2 \in Y$, then  $Y=\{y_1,y_2,x_3,\dots,x_k\}$ and so $\edge{B_1,x_1,y_1,Y}$.	Likewise, if $y_3 \in Y$ then $\edge{C_1,x_1,y_1,Y}$.  \label{prop:c:thet233:d}
			
		\end{enumerate}

		From \ref{prop:c:thet233:a} and \ref{prop:c:thet233:b}, $y_1 \in Y$ but $y_1 \not\in B_2$, and similarly from \ref{prop:c:thet233:d} $y_2 \in B_2$ but $y_2 \not\in Y$; therefore,
		$\edge{B_2,y_2,y_1,Y}$.   Now, since $x_1 \sim y_1$, and $y_1 \in Y$, we have that $x_1 \not\in Y$. Thus, if $x_1$ were in $B_2$, its token would move in the transition from $B_2$ to $Y$.  However,  have already established that it is the token at $y_2$ that moves, and so $x_1 \not\in B_2$.  We conclude that $\edge{B_1,x_1,z,B_2}$ for some $z \not\in \{y_1,y_3\}$, so that $B_2 = \{z,y_2,x_3,\dots,x_k\}$.  Notice that since $B_2$ is independent, and $x_2 \sim y_2$, it follows that $z \neq x_2$.
		
		Using similar arguments, we determine that $\edge{C_2,y_3,y_1,Y}$, and that $\edge{C_1,x_1, w,C_2}$ for some $w\notin \{y_1,y_2,x_1\}$.  Moreover, $C_2=\{w,x_2,y_3,x_4,\dots,x_k\}$.  Again, note that since $x_3 \sim y_3$, $w \neq x_3$.
		
		From $\edge{B_2,y_2,y_1,Y}$, we have that $Y=\{y_1,z,x_3,x_4\dots,x_k\}$.  However, from $\edge{C_2,y_3,y_1,Y}$, we also have that $Y=\{y_1,x_2,w,x_4,\dots,x_k\}$.  As we have already established that $z \neq x_2$ and $w \neq x_3$, we arrive at two contradicting compositions of $Y$.
		Thus, no such graph $G$ exists, and we conclude that $\thet{2,3,3}$ is not an $i$-graph.
	\end{proof}
	

	\subsection{ $\thet{2,3,4}$ is not an $i$-Graph} \label{subsubsec:c:234}

	\begin{prop}	\label{prop:c:thet234}
		The graph $\thet{2,3,4}$ is not $i$-graph realizable.
	\end{prop}
	
	\begin{figure}[H]
		\centering
		\begin{tikzpicture}
			
			\node(cent) at (0,0) {};			
			\path (cent) ++(0:60 mm) node(rcent)  {};		
			
			\foreach \i/\j/\c in {x/X/90,y/Y/-90} {				
				\path (cent) ++(\c:20 mm) node [std] (v\i) {};
			}
			
			\path (vx) ++(90:5 mm) node (Lvx) {$\color{blue}X=\{x_1,x_2,x_3,\dots, x_k\}$};
			\path (vy) ++(-90:5 mm) node (Lvy) {$\color{red}Y = \{y_1,z,x_3,\dots, x_k\}$};
			
			\path (cent) ++(180:8 mm) node [std] (va) {};

			\path (va) ++(175:20 mm) node (Lvau) {$\color{blue}A=$};
			\path (va) ++(185:20 mm) node (Lvad) {$\color{blue}\{y_1,x_2,x_3,\dots, x_k\}$};

			\draw[thick] (vx)--(va)--(vy);

			
			\path (cent) ++(0:8 mm) node (vb) {};
			
			\path (vb) ++(90:8 mm) node [std] (vb1) {};
			\path (vb1) ++(0:25 mm) node (Lvb) {\color{blue}$B_1=\{x_1,y_2,x_3,\dots, x_k\}$};
			
			\path (vb) ++(90:-8 mm) node [std] (vb2) {};
			\path (vb2) ++(0:25 mm) node (Lvb) {\color{red}$B_2=\{z, y_2,x_3,\dots, x_k\}$};	
			
			\draw[thick](vx)--(vb1)--(vb2)--(vy);

			\foreach \i/\c in {1/15,2/0,3/-15} {				
				\path (rcent) ++(90:\c mm) node [std] (c\i) {};
			}
			
			\draw[thick] (vx)--(c1)--(c2)--(c3)--(vy);
			
			
			\path (c1) ++(0:25 mm) node (Lvc1) {\color{blue}$C_1=\{x_1,x_2,y_3,\dots, x_k\}$};
			
			\path (c2) ++(0:5 mm) node (Lvc2) {\color{red}$C_2$};
			
			
			\path (c3) ++(0:5 mm) node (Lvc3) {\color{red}$C_3$};

		\end{tikzpicture}				
		\caption{$H=\thet{2,3,4}$ non-construction.}
		\label{fig:c:thet234}%
	\end{figure}

	\begin{proof}
		The construction for our contradiction begins similarly to that of $\thet{2,3,3}$ in Proposition \ref{prop:c:thet233}.  As before, we illustrate the graph  in Figure \ref{fig:c:thet234}, labelling the known sets in blue, and those yet to be determined in red.  Given the similarity of $\thet{2,3,3}$ and $\thet{2,3,4}$, many of the observations from Proposition \ref{prop:c:thet233} carry through to our current proof.  In particular, all of \ref{prop:c:thet233:a} - \ref{prop:c:thet233:d} hold here, including that $y_1\notin C_2$ from  \ref{prop:c:thet233:b}.  Moreover, the compositions of $Y$ and $B_2$ also hold, where $z$ is some vertex with $z \not\in \{y_1,x_2,y_2\}$.
		
		We now attempt to build $C_2$.  From Proposition \ref{prop:i:C4struct}, since $X \not\sim C_2$, $y_3 \in C_2$ (and $x_3 \not\in C_2$).  From the distance requirement of Observation \ref{obs:i:dist}, $|X - C_2| \leq 2$, and so at least one of $x_1$ or $x_2$ is in $C_2$.  Recall from the construction for Proposition \ref{prop:c:thet233} that $\edge{A,x_2,z,Y}$ and $\edge{B_1,x_1,z,B_2}$, and so $z$ is adjacent to both $x_1$ and $x_2$.  Hence, $z \not\in C_2$.  
		
		Gathering these results shows that none of $\{x_3,z,y_1\}$ are in $C_2$, and thus, $d(C_2, Y) \geq 3$, contradicting the distance requirement of Observation \ref{obs:i:dist}.  We conclude that no graph $G$ exists such that $\ig{G} = \thet{2,3,4}$.
	\end{proof}
	

	
	\subsection{$\thet{3,3,3}$ is not an $i$-Graph}  \label{subsubsec:c:333}

	\begin{prop}	\label{prop:c:thet333}
		The graph $\thet{3,3,3}$ is not $i$-graph realizable.
	\end{prop}
	
	\begin{figure}[H]
		\centering
		\begin{tikzpicture}
			
			\node(cent) at (0,0) {};			
			\path (cent) ++(0:60 mm) node(rcent)  {};		
			
			\foreach \i/\j/\c in {x/X/90,y/Y/-90} {				
				\path (cent) ++(\c:20 mm) node [std] (v\i) {};
			}
			
			\path (vx) ++(90:5 mm) node (Lvx) {$\color{blue}X=\{x_1,x_2,x_3,\dots, x_k\}$};
			\path (vy) ++(-90:5 mm) node (Lvy) {$\color{red}Y$};
			
			\path (cent) ++(180:8 mm) node (va) {};	
			
			
			\path (va) ++(90:8 mm) node [std] (va1) {};
			\path (va1) ++(180:25 mm) node (Lva1) {\color{blue}$A_1=\{y_1,x_2,x_3,\dots, x_k\}$};
			
			\path (va) ++(90:-8 mm) node [std] (va2) {};
			\path (va2) ++(180:5 mm) node (Lva2) {\color{red}$A_2$};

			\draw[thick] (vx)--(va1)--(va2)--(vy);

			
			\path (cent) ++(0:8 mm) node (vb) {};
			
			\path (vb) ++(90:8 mm) node [std] (vb1) {};
			\path (vb1) ++(0:25 mm) node (Lvb) {\color{blue}$B_1=\{x_1,y_2,x_3,\dots, x_k\}$};
			
			\path (vb) ++(90:-8 mm) node [std] (vb2) {};
			\path (vb2) ++(0:5 mm) node (Lvb) {\color{red}$B_2$};	
			
			\draw[thick](vx)--(vb1)--(vb2)--(vy);

			\foreach \i/\c in {1/14,2/-14} {				
				\path (rcent) ++(90:\c mm) node [std] (c\i) {};
			}
			
			\draw[thick] (vx)--(c1)--(c2)--(vy);
			
			
			\path (c1) ++(0:25 mm) node (Lvc1) {\color{blue}$C_1=\{x_1,x_2,y_3,\dots, x_k\}$};
			\path (c2) ++(0:5 mm) node (Lvc2) {\color{red}$C_2$};

		\end{tikzpicture}				
		\caption{$H=\thet{3,3,3}$ non-construction.}
		\label{fig:c:thet333}%
	\end{figure}

	\begin{proof}
		Let $H$ be the theta graph $\thet{3,3,3}$, with vertices labelled as in Figure \ref{fig:c:thet333}.  Suppose to the contrary that there exists some graph $G$ such that $H$ is the $i$-graph of $G$;  that is, $\ig{G} = \thet{3,3,3}$.
		
		
		
		Since $d_H(A_1,Y)=2$,  by Observation  \ref{obs:i:d2}, $|A_1 - Y|=2$.  Similarly, $|B_1-Y|=2$ and $|C_1 - Y|=2$.  
		Suppose that, say, $x_4 \not\in Y$. Hence, by Observation \ref{obs:i:dist} $|\{x_1,x_2,x_3\} \cap Y | \geq 1$.  Without loss of generality, say $x_1 \in Y$.    Then since $y_1 \not\in Y$ and $x_4 \notin Y$, both $x_2$ and $x_3 \in Y$ to satisfy $|A_1 - Y|=2$.  However, then $Y=\{x_1,x_2,x_3,z,x_5,\dots,x_k\}$ for some vertex $z \sim x_4$, and so $\edge{X,x_4,z,Y}$, which is not so.  	
		We therefore conclude that $x_4 \in Y$, and likewise $x_i \in Y$ for $i\geq 4$.  Thus, $\{x_1,x_2,x_3\} \cap Y = \varnothing$.
		
		Returning to $A_1$, since $d(A_1,Y)= 2$ and  $x_2,x_3 \notin Y$,  we have that $y_1 \in Y$.  Similarly, $y_2, y_3 \in Y$.  Thus, $Y=\{y_1,y_2,y_3,x_3,\dots,x_k\}$.  Moreover, $A_2$ is obtained from $A_1$ by replacing one of $x_2$ or $x_3$, by $y_2$ or $y_3$,  respectively.  Say, $\edge{A_1,x_2,y_2,A_2}$ so that $A_2=\{y_1,y_2,x_3,\dots,x_k\}$.  Now, however, we have that $\edge{B_1,x_1,y_1,A_2}$, but clearly $B_1 \not\sim A_2$.  It follows that $\thet{3,3,3}$ is not an $i$-graph.
	\end{proof}
	

	This completes the proof of Theorem \ref{thm:c:thetas}.

	
	\section{Other Results}  \label{sec:c:other}
	In this section we first display a graph that is neither a theta graph nor an $i$-graph, and then use the method of graph complements to show that every cubic 3-connected bipartite planar graph is an $i$-graph.

	\subsection{A Non-Theta Non-$i$-Graph}  \label{subsubsec:c:nonThetnonI}

	So far, every non-$i$-graph we have observed is either one of the seven theta graphs from Theorem \ref{thm:c:thetas}, or contains one of those seven  as an induced subgraph (as per Corollary \ref{coro:i:notInduced}).  This leads naturally to the question of whether theta graphs provide a forbidden subgraph characterization for $i$-graphs.  Unfortunately, this is not the case.  
	
	Consider the graph $\ntni$ in Figure \ref{fig:c:nonThetnonI}: it is not a theta graph, and although it contains several theta graphs as induced subgraphs, none of those induced subgraphs are among the seven non-$i$-graph theta graphs.  In Proposition \ref{prop:c:nonThetnonI} we confirm  that $\ntni$ is not an $i$-graph.

	\begin{figure}[H]
		\centering
		\begin{tikzpicture}
			
			\node(cent) at (0,0) {};			
			\path (cent) ++(0:65 mm) node(rcent)  {};		
			
			\foreach \i/\j/\c in {x/X/90,y/Y/-90} {				
				\path (cent) ++(\c:20 mm) node [std] (v\i) {};
			}
			
			\path (vx) ++(90:5 mm) node (Lvx) {$\color{blue}X=\{x_1,x_2,x_3,x_4\dots, x_k\}$};
			\path (vy) ++(-90:5 mm) node (Lvy) {$\color{red}Y=\{y_1,y_2,y_3,x_4,\dots,x_k\}$};
			
			\path (cent) ++(180:8 mm) node (va) {};	
			
			
			\path (va) ++(90:8 mm) node [std] (va1) {};
			\path (va1) ++(180:25 mm) node (Lva1) {\color{blue}$A_1=\{y_1,x_2,x_3,\dots, x_k\}$};
			
			\path (va) ++(90:-8 mm) node [std] (va2) {};
			\path (va2) ++(180:25 mm) node (Lva2) {\color{red}$A_2=\{y_1,x_2,y_3,\dots,x_k\}$};

			\draw[thick] (vx)--(va1)--(va2)--(vy);

			
			\path (cent) ++(0:8 mm) node (vb) {};
			
			\path (vb) ++(90:8 mm) node [std] (vb1) {};
			\path (vb1) ++(0:25 mm) node (Lvb) {\color{blue}$B_1=\{x_1,y_2,x_3,\dots, x_k\}$};
			
			\path (vb) ++(90:-8 mm) node [std] (vb2) {};
			\path (vb2) ++(0:25 mm) node (Lvb) {\color{blue}$B_2=\{y_1,y_2,x_3,\dots, x_k\}$};	
			
			\draw[thick](vx)--(vb1)--(vb2)--(vy);

			\foreach \i/\c in {1/14,2/0,3/-13} {				
				\path (rcent) ++(90:\c mm) node [std] (c\i) {};
			}
			
			\draw[thick] (vx)--(c1)--(c2)--(c3)--(vy);
			
			
			\path (c1) ++(0:5 mm) node (Lvc1) {\color{red}$D_1$};
			\path (c2) ++(0:5 mm) node (Lvc2) {\color{red}$D_2$};
			\path (c3) ++(0:5 mm) node (Lvc3) {\color{red}$D_3$};

			\draw[thick](va1)--(vb2);
			
		\end{tikzpicture}				
		\caption{A non-theta non-$i$-graph $\ntni$.}
		\label{fig:c:nonThetnonI}%
	\end{figure}
	
	\begin{prop} \label{prop:c:nonThetnonI}
		The graph $\ntni$ in Figure \ref{fig:c:nonThetnonI} is not an $i$-graph.
	\end{prop}

	\begin{proof}
		We proceed similarly to the proofs for the theta non-$i$-graphs in Section \ref{sec:c:nonthetas}.   Let $\ntni$ be the graph in Figure \ref{fig:c:nonThetnonI}, with vertices as labelled, and suppose to the contrary that there is some graph $G$ such that $\ig{G}\cong \ntni$.  
		
		To begin, we determine the vertices of the two induced $C_4$'s of $\ntni$.  Immediately from Proposition \ref{prop:i:C4struct}, the vertices of $X$, $A_1$, $B_1$, and $B_2$ are as labelled in Figure \ref{fig:c:nonThetnonI}.  Using a second application of Proposition \ref{prop:i:C4struct}, $A_2$ differs from $A_1$ in exactly one position, different from $B_2$.  Without loss of generality, say, $\edge{A_1, x_3, y_3, A_2}$.  Then again by the proposition, $Y=\{y_1, y_2, y_3, x_4 \dots, x_k\}$ as in Figure \ref{fig:c:nonThetnonI}.
		
		We now construct the vertices of $D_1$ through a series of claims:
		
		
		
		\begin{enumerate}[label=(\roman*)]
			\item \label{prop:c:nonThetnonI:i} $x_3$ is not in $D_1$. 
			
			From Lemma \ref{lem:i:claw}, $D_1$ differs from $X$ in one vertex, different from both $A_1$ and $B_1$; moreover,  $x_1$ and $x_2$ are in $D_1$.    		
			Notice that $\{x_4,x_5,\dots,x_k\} \subseteq D_1$: suppose to the contrary that, say,  $x_4 \notin D_1$, so that $\edge{X, x_4,z,D_1}$.  Then $z \neq y_1$, since $y_1 \sim_G x_1$ and $D_1$ is independent.  Likewise $z \neq y_2$ and $z \neq y_3$. Thus $D_1 = \{x_1, x_2, x_3, z, x_5,\dots, x_k\}$ has $|Y \cap D_1| =4$, but $d(D_1,Y)=3$, contradicting Observation \ref{obs:i:d2}.

			\item \label{prop:c:nonThetnonI:ii} $y_1$, $y_2,$ and $y_3$ are not in $D_1$.

			As above, $y_1$ and $y_2$ are not in $D_1$, as $D_1$ is an independent set containing $x_1$ and $x_2$.  
			Suppose to the contrary that $\edge{X,x_3,y_3,D_1}$, so that $D_1 = \{x_1, x_2, y_3,  \dots, x_k\}$. Then, since $x_1 \sim_G y_1$, we have that $\edge{D_1,x_1,y_1,A_2}$, a contradiction.
			
			\item \label{prop:c:nonThetnonI:iii} $D_1 = \{x_1, x_2, z, x_4, \dots, x_k\}$, where $z \notin \{y_1,y_2,y_3\}$. 
			
			Immediate from \ref{prop:c:nonThetnonI:i} and \ref{prop:c:nonThetnonI:ii}.

		\end{enumerate}
		
		A contradiction for the existence of $G$ arises as we construct $D_2$.  Since $d_{\ntni}(D_1,Y)=3$ and $|D_1 \cap Y| = 3$, at each step along the path through $D_2, D_3,$ and $Y$, exactly one token departs  from a vertex of $D_1$  and moves to one of $Y-D_1=\{y_1,y_2,y_3\}$.
		However, $D_2 \neq \{y_1,x_2,z,x_4,\dots,x_k\}$, since otherwise, $\edge{A_1,x_3,z,D_2}$.  Likewise, $D_2 \neq \{x_1,y_2,z,x_4,\dots,x_k\}$.  Finally $D_2 \neq \{x_1,x_2,y_3,x_4,\dots,x_k\}$ as again we have $\edge{A_2,y_1,x_1,D_2}$.  Thus, we cannot build a set $D_2$ such that $|D_2 \cap Y| \leq 2$ as required.  We so conclude that no such graph $G$ exists, and that $\ntni$ is not an $i$-graph.
	\end{proof}
	
	Like the seven non-$i$-graph realizable theta graphs of Theorem \ref{thm:c:thetas},  $\ntni$ is a minimal obstruction to a graph being an $i$-graph; every induced subgraph of  $\ntni$ is a theta graph.

	\subsection{Maximal Planar Graphs} \label{subsubsec:c:maxPlanar}

	We conclude this section with a result to demonstrate that certain planar graphs are $i$-graphs and $\agraph$-graphs.  Our proof uses the following three known results.

	
	\begin{theorem} \emph{\cite[Theorem 4.6]{CLZ}} \label{thm:c:ref:cub3con}  A cubic graph is $3$-connected if and only if it is $3$-edge connected.
	\end{theorem}

	\begin{theorem} \emph{\cite{W69}} \label{thm:c:ref:3col}  
		A connected planar graph $G$ is bipartite if and only if its dual $\dual{G}$ is Eulerian.
	\end{theorem}

	\begin{theorem} \emph{\cite{H1898, TW11}} \label{thm:c:ref:maxPlan}  
		A maximal planar graph $G$ of order at least $3$ has $\chi(G) = 3$ if and only if ${G}$ is Eulerian.
	\end{theorem}

	In our proof of the following theorem, we consider a graph $G$ that is cubic, 3-connected, bipartite, and planar. We then examine its dual $\dual{G}$, and how those specific properties of $G$ translate to $\dual{G}$.  Then, we construct the complement of $\dual{G}$, which we refer to as $H$.  We claim that $H$ is a seed graph of $G$; that is, $\ig{H}$ contains an induced copy of $G$.

	\begin{theorem} \label{thm:c:3con}  
		Every cubic $3$-connected bipartite planar graph is an $i$-graph and an $\agraph$-graph.	
	\end{theorem}
	
	\begin{proof}
		Let $G$ be a cubic 3-connected bipartite planar graph and consider the dual $\dual{G}$ of $G$.  Since $G$ is bridgeless, $\dual{G}$ has no loops.  Moreover, since $G$ is 3-connected, Theorem \ref{thm:c:ref:cub3con} implies that no two edges separate $G$; hence, $\dual{G}$ has no multiple edges.  Therefore, $\dual{G}$  is a (simple) graph.  Further, since $G$ is cubic, each face of $\dual{G}$ is a triangle, and so $\dual{G}$ is a maximal planar graph.  
		
		We note the following two key observations.  First, since each face of $\dual{G}$ is a triangle,

		\begin{enumerate}[label=(\arabic*)]
			\item  \label{thm:c:3con:i}   each edge of $\dual{G}$ belongs to a triangle.
			
		\end{enumerate}
		
		\noindent For the second, note that since $G$ is bipartite, $\dual{G}$ is Eulerian (by Theorem \ref{thm:c:ref:3col}).  Then, by Theorem \ref{thm:c:ref:maxPlan}, $\chi(\dual{G})=3$, so $\dual{G}$ does not contain a copy of $K_4$.  Thus,
		
		\begin{enumerate}[label=(\arabic*),resume]
			\item \label{thm:c:3con:ii}   every triangle of $\dual{G}$ is a maximal clique.
		\end{enumerate}

		Now, by the duality of $\dual{G}$ and $G$, there is a one-to-one correspondence between the facial triangles of $\dual{G}$ and the vertices of $G$.   Let $H$ be the complement of $\dual{G}$.  By \ref{thm:c:3con:i} and \ref{thm:c:3con:ii}, $i(H) = \alpha(H)=3$, and every maximal independent set of $G$ corresponds to a triangle of $\dual{G}$.  But again, by duality, the facial triangles of $\dual{G}$ and their adjacencies correspond to the vertices of $G$ and their adjacencies.  Therefore, $\ig{H}$ contains $G$ as an induced subgraph.   Any additional unwanted vertices of $\ig{H}$ can be removed by applying the Deletion Lemma (Lemma \ref {lem:i:inducedI}).
	\end{proof}

\section{Open Problems}
\label{sec:open}
We conclude with a few open problems. Although the problems are stated here for $i$-graphs, many are relevant to other reconfiguration graphs pertaining to domination-type parameters and are also mentioned in \cite{MT18, LauraD}.

\begin{problem}
	\emph{Determine conditions on the graph $G$ under which $\ig{G}$ is (a) connected (b) disconnected.}
\end{problem}	

A graph $G$ is \emph{well-covered} if all its maximal independent sets have the same cardinality, that is, if $\alpha(G)=i(G)$.

\begin{problem}
    \begin{enumerate}
       \item[\emph{(a)}] \emph{Determine those $i$-graphs that are also $\alpha$-graph realizable.}
        \item[\emph{(b)}] \emph{Determine those $i$-graphs $H$ for which there exists a well-covered seed graph $G$ such that $\ig{G}\cong H$. }
    \end{enumerate}
	
\end{problem}

\begin{problem} 
	\emph{We have already seen that classes of graphs trees, cycles, and, more generally, block graphs, are $i$-graphs. \smallskip
 \newline
  As we build new graphs from these families of $i$-graphs, using tree structures, which of those are also $i$-graphs?  For example, are cycle-trees $i$-graphs?  Path-trees?  For what families of graphs $H$ are $H$-trees $i$-graphs?}
\end{problem}

\begin{problem} \label{con:op:tree}
	\emph{Determine the structure of  $i$-graphs of various families of trees.  For example, consider 
 \begin{enumerate}
    \item[(a)] caterpillars in which every vertex has degree $1$ or $3$,
    \item[(b)] spiders ($K_{1,r}$ with each edge subdivided).
    \end{enumerate}}

\end{problem}

\begin{problem} \label{con:op:ham}
	\emph{Find more classes of $i$-graphs that are Hamiltonian, or Hamiltonian traceable.}
\end{problem}

\begin{problem} \label{con:op:domChain}
	\emph{Suppose $G_1, G_2,\dots$ are graphs such that $\ig{G_1} \cong G_2$, $\ig{G_2}\cong G_3$, $\ig{G_3} \cong G_4, \dots$ .  Under which conditions does there exist an integer $k$ such that $\ig{G_k} \cong G_1$?}
	
\end{problem}

As a special case of Problem \ref{con:op:domChain}, note that for any $n\geq 1$, $\ig{K_n} \cong K_n$, and that for $k \equiv 2 \Mod{3}$, $\ig{C_k} \cong C_k$.   

\begin{problem} 
	\emph{Characterize the graphs $G$ for which $\ig{G} \cong G$.}
\end{problem}

\medskip

\bigskip

\noindent\textbf{Acknowledgements\hspace{0.1in}}We acknowledge the support of
the Natural Sciences and Engineering Research Council of Canada (NSERC), RGPIN-2014-04760 and RGPIN-03930-2020.

\noindent Cette recherche a \'{e}t\'{e} financ\'{e}e par le Conseil de
recherches en sciences naturelles et en g\'{e}nie du Canada (CRSNG), RGPIN-2014-04760 and
RGPIN-03930-2020.
\begin{center}
\includegraphics[width=2.5cm]{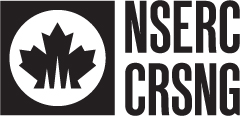}%
\end{center}
	

	\bibliographystyle{abbrv} 
	\bibliography{LTPhDBib} 
	
\end{document}